\documentclass[11pt]{amsart}

\usepackage{amsmath,amssymb,amscd,amsthm,amsxtra,amsfonts}
\usepackage[dvips]{graphics,epsfig}
\usepackage[all]{xy}
\usepackage{url}

\usepackage{color}

\allowdisplaybreaks[2]

\newtheorem{theorem}{Theorem} [section]

\newtheorem{lemma}[theorem]{Lemma}
\newtheorem{proposition}[theorem]{Proposition}
\newtheorem{remark}[theorem]{Remark} 

\newtheorem{definition}{Definition}

\DeclareMathOperator*{\supp}{supp}

\newcommand{\noi}{\noindent}
\newcommand{\N}{\mathbb{N}}
\newcommand{\M}{\mathcal{M}}
\newcommand{\Z}{\mathbb{Z}}
\newcommand{\R}{\mathbb{R}}
\newcommand{\T}{\mathbb{T}}
\newcommand{\C}{\mathbb{C}}

\newcommand{\Ll}{\mathcal{L}}
\newcommand{\Nn}{\mathcal{N}}

\newcommand{\g}{\mathfrak{g}}
\newcommand{\oet}{(\omega|_{M})_{\eta}}
\newcommand{\oget}{(\omega|_{M})_{g\cdot\eta}}

\newcommand{\dl}{\delta}

\newcommand{\eps}{\varepsilon}

\newcommand{\cj}{\overline}

\newcommand{\dt}{\partial_t}

\newcommand{\jb}[1]
{\langle #1 \rangle}

\numberwithin{equation}{section}
\numberwithin{theorem}{section}

\begin{document}

\title
[Soliton interaction with small Toeplitz potentials]
{\bf Soliton interaction with small Toeplitz
potentials for the Szeg\"o equation on $\R$}

\author{Oana Pocovnicu}
\address{Oana Pocovnicu\\
Laboratoire de Math\'ematiques d'Orsay\\
Universit\'e Paris-Sud (XI)\\
91405, Orsay Cedex, France}

\email{oana.pocovnicu@math.u-psud.fr}

\subjclass[2000]{ 35B15, 35Q51, 37K40, 37K10, 47B35.}

\keywords{Szeg\"o equation; soliton; effective Hamiltonian; Toeplitz operators}

\begin{abstract}
We consider the cubic Szeg\"{o} equation with a small Toeplitz potential
and with soliton initial data
\begin{equation*}
\begin{cases}
i\dt u=\Pi(|u|^{2}u)+\eps T_bu\\
u(0,x)=\alpha_0e^{i\phi_0}\mu_0\eta(\mu_0(x-a_0)).
\end{cases}
\end{equation*}

\noi
We show that up to time $\eps^{-1/2}\log(1/\eps)$ and errors of size $\eps^{1/2}$, the solution preserves the soliton
shape
$u(t,x)=\alpha e^{i\phi}\mu\eta(\mu(x-a))$, and the time dependent parameters
$a,\alpha,\phi,\mu$ evolve according to the effective dynamics, up to small corrections.
\end{abstract}

\date{\today}
\maketitle


\section{Introduction}

One of the most important properties
in the study of the nonlinear Schr\"odinger equations (NLS) is {\it dispersion}.
It is often exhibited in the form of the Strichartz estimates
of the corresponding linear flow.
In case of the cubic NLS:
\begin{equation}\label{eqn: Schrodinger 3}
i\partial_t u+\Delta u=|u|^2u, \quad (t,x)\in\R\times M,
\end{equation}

\noi
G\'erard and Grellier \cite{PGSGX} remarked that
there is a lack of dispersion
when $M$ is a sub-Riemannian manifold
(for example, the Heisenberg group).
In this situation, many of the classical arguments
used in the study of NLS no longer hold.
As a consequence, even the problem of global well-posedness of \eqref{eqn: Schrodinger 3}
on a sub-Riemannian manifold still remains open.
In \cite{PGSG, PGSGX},
G\'erard and Grellier introduced a model of a non-dispersive Hamiltonian equation
called {\it the cubic Sz\"ego equation.} (See \eqref{eq:szego 3} below.)
The study of this equation is expected to give new tools to be used in understanding
existence and other properties of smooth solutions of NLS in the absence of dispersion.

In this paper we will consider the Szeg\"o equation on the real line. The space
of solutions in this case is the Hardy space $L^2_+(\R)$
on the upper half-plane \\ $\C_+ = \{ z; \text{Im} z > 0\}$,
defined by
\[L^2_+(\R)=\{f\in L^2(\R); \, \supp{\hat{f}}\subset [0,\infty)\}.\]

\noi
The corresponding Sobolev spaces $H^s_+(\R)$, $s\geq 0$ are defined by:
\begin{align*}
H^s_+(\R)=&\big\{h\in L^2_+(\R);  \|h\|_{H^s_+}:
=\bigg{(}\frac{1}{2\pi}\int_0^{\infty}(1+|\xi|^2)^{s}|\hat{h}(\xi)|^2d\xi\bigg{)}^{1/2}<\infty\big\}.
\end{align*}

\noi
The Szeg\"o projector $\Pi$ is the projector on the non-negative frequencies,
\\ $\Pi:L^{2}(\R)\to L^2_+(\R)$
\[\Pi (f)(x)=\frac{1}{2\pi}\int_{0}^{\infty}e^{ix\xi}\hat{f}(\xi)d\xi.\]

\noi
For $u \in L^2_+(\R)$,
we consider \textit{the Sz\"ego equation on the real line}:
\begin{equation}\label{eq:szego 3}
i\dt u=\Pi(|u|^{2}u), \quad (t, x) \in \R\times \R.
\end{equation}

\noi
This equation is globally well-posed
in $H^{\frac{1}{2}}_+(\R)$.

On $L^2_+(\R)$ we introduce the symplectic form
\begin{equation*}
\omega(u,v)=\textup{Im}\int_{\R}u\bar{v}dx
\end{equation*}

\noi
and the real scalar product
\begin{equation*}
\jb{u,v}=\textup{Re}\int_{\R}u\bar{v}dx.
\end{equation*}

Let $\mathcal{D}\subset L^2_+(\R)$ be a dense subset of
$L^2_+(\R)$. We say that a function $F:\mathcal{D}\to\R$ admits a
Hamiltonian vector field $X_F:\mathcal{D}\to L^2_+(\R)$ if
\begin{equation*}
d_uF(h)=\omega(h,X_{F}(u)),
\end{equation*}

\noi
for all $u,h\in \mathcal{D}$. The function
\begin{equation*}
H(u)=\frac{1}{4}\int_{\R}|u(x)|^4dx
\end{equation*}

\noi
defined on $L^4_+(\R)$, admits the Hamiltonian vector field
\begin{equation*}
X_H(u)=-i\Pi (|u|^2u),
\end{equation*}

\noi
Thus the Szeg\"o equation is a Hamiltonian evolution.
The most remarkable property of this equation is the fact that it is completely integrable
in the sense that it posses a Lax pair structure \cite{pocov1}. The Lax pair is given in terms of
{\it Hankel
and Toeplitz operators}.

A Hankel operator $H_{u}:L^2_+\to L^2_+$ of symbol $u\in H^{1/2}_+$ is defined by
\[H_{u}(h)=\Pi(u\bar{h}).\]

\noi
$H_u$ is a Hilbert-Schmidt operator, it is $\C$-anti-linear and satisfies
\begin{equation}\label{sym H_u 3}
(H_{u}(h_{1}),h_{2})=(H_{u}(h_{2}),h_{1}).
\end{equation}

\noi
A Toeplitz operator $T_{b}:L^2_+\to L^2_+$ of symbol $b\in L^{\infty}(\R)$ is defined by
\[T_{b}(h)=\Pi(bh).\]

\noi
$T_{b}$ is $\C$-linear and bounded.  Moreover, $T_b$ is self-adjoint if and only if $b$ is real-valued.

In what follows we consider the perturbed Szeg\"o equation
with a small Toeplitz potential
\begin{equation}\label{PS}
i\partial_t u=\Pi (|u|^2u)+\eps T_bu.
\end{equation}

\noi
This is no longer a completely integrable equation. It is still
globally well posed in $H^{\frac{1}{2}}_+(\R)$ if $b\in H^1(\R)$. This can
be proved by following the lines of the proof of Theorem 2.1 in \cite{PGSG}
on the global well-posedness of the Szeg\"o equation.

If instead of the
Toeplitz potential we considered a multiplicative
linear potential $bu$,
then the corresponding equation would no longer be Hamiltonian. However, if
we project to $L^2_+$, obtaining this way a Toeplitz potential $T_bu=\Pi(bu)$,
we conserve the Hamiltonian structure of the Szeg\"o equation. For this reason,
the Toeplitz potential is the natural generalization
of the linear multiplicative potential
in the case of the Szeg\"o equation.

The Hamiltonian of equation \eqref{PS} is
\begin{equation*}
H_b(u)=\frac{1}{4}\int_{\R}|u(x)|^4dx+\frac{\eps}{2}\int_{\R}b(x)|u(x)|^2dx.
\end{equation*}

\noi
This yields that the Hamiltonian $H_b$ is formally
conserved by the flow. Note also that the fact that $b$
is a real valued function, yields the conservation of the mass
$Q(u)=\int |u|^2dx.$

The goal of the paper is to study the long time
behavior
of the solution of the perturbed Szeg\"o equation \eqref{PS}
having as initial condition a soliton of the unperturbed equation.
\begin{definition}
A soliton for the Szeg\"o equation on the real line is a solution
$u$ with the property that there exist $c,\omega\in\R$, $c\neq 0$ such that
\[u(t,x)=e^{-it\omega}u_0(x-ct).\]
\end{definition}

In \cite[Theorem 2]{pocov1}
it was proved that all the initial data of solitons for the Szeg\"o
equation on $\R$ are of the form
\begin{equation}\label{eqn: soliton}
u_0=e^{i\phi_0}\alpha_0 \mu_0 \eta(\mu_0(x-a_0))=\frac{e^{i\phi_0}\alpha_0}{x-a_0+\frac{i}{\mu_0}},
\end{equation}

\noi
where $\eta(x):=\frac{1}{x+i}$, $\alpha_0,\mu_0\in (0,\infty)$,
and $\phi_0, a_0\in\R$, and that the corresponding solution is
\begin{equation}\label{eqn: soliton at t}
u(t,x)=e^{i\phi(t)}\alpha_0 \mu_0 \eta(\mu_0(x-a(t)))=\frac{e^{i\phi(t)}\alpha_0}{x-a(t)+\frac{i}{\mu_0}},
\end{equation}

\noi
where  $\phi(t)=-\frac{\alpha_0^2\mu_0^2}{4}t+\phi_0$ and $a(t)=\frac{\alpha_0^2\mu_0}{2}t+a_0$.

We show that the solution of the perturbed Szeg\"o equation \eqref{PS} with initial data
$u_0=e^{i\phi_0}\alpha_0 \mu_0 \eta(\mu_0(x-a_0))$ preserves the form $u=e^{i\phi}\alpha \mu \eta(\mu(x-a))$
over a large interval of time, and the time dependent parameters $a,\alpha,\phi,\mu$ evolve according to the
effective dynamics, up to small corrections. More precisely, the main result of the paper is the following theorem.
\begin{theorem}\label{main theorem}
Let $b:\R\to\R$ be a function in $H^1(\R)$
with the property that $b'\in L^1(\R)$.
 Let $0<\eps\ll 1$ and $0<\delta<\frac{1}{2}$.
 If $u$ is a solution of the perturbed Szeg\"o equation with a small Toeplitz potential
\begin{align}\label{Szego T}
\begin{cases}
i\partial_t u=\Pi (|u|^2u)+\eps T_bu\\
u(0,x)=\alpha_0e^{i\phi_0}\mu_0\eta(\mu_0(x-a_0)),
\end{cases}
\end{align}

\noi
where $a_0,\phi_0\in\R$ and $\alpha_0,\mu_0\in (0,\infty)$,
 then
\begin{align*}
\|u(t)-\alpha(t)e^{i\phi(t)}\mu(t)\eta(\mu(t)(x-a(t)))\|_{H^{\frac{1}{2}}_+}
\leq C\eps^{\frac{1}{2}+\frac{\delta}{3}},
\end{align*}

\noi
for times $0\leq t\leq \frac{\dl}{6\ln c_0}\cdot\frac{1}{\eps^{\frac{1}{2}-\delta}}\ln(\frac{1}{\eps})$,
 where $c_0$ is a constant depending only on $\alpha_0$ and $\mu_0$, and $a,\alpha,\phi,\mu$ satisfy
\begin{align}\label{sys perturbations}
\begin{cases}
\dot{a}=\frac{\alpha^2\mu}{2}-\frac{2\eps}{\pi\mu}
\int b'(a+\frac{x}{\mu})\frac{x}{\mu}|\eta(x)|^2dx+O(\eps^{1+\frac{2\delta}{3}}),\\
\dot{\alpha}=\frac{\eps\alpha}{\pi\mu}
\int b'(a+\frac{x}{\mu})|\eta(x)|^2dx+O(\eps^{1+\frac{2\delta}{3}}),\\
\dot{\phi}=-\frac{\alpha^2\mu^2}{4}-\frac{\eps}{\pi}
\int b(a+\frac{x}{\mu})|\eta(x)|^2dx-\frac{\eps}{\pi}
\int b'(a+\frac{x}{\mu})\frac{x}{\mu}|\eta(x)|^2dx+O(\eps^{1+\frac{2\delta}{3}}),\\
\dot{\mu}=-\frac{2\eps}{\pi}
\int b'(a+\frac{x}{\mu})|\eta(x)|^2dx+O(\eps^{1+\frac{2\delta}{3}}).
\end{cases}
\end{align}

\noi
In addition,
if $\bar{a}, \bar{\alpha}, \bar{\phi}, \bar{\mu}$ satisfy
\begin{align}\label{effective dynamics}
\begin{cases}
\dot{\bar{a}}=\frac{\bar{\alpha}^2\bar{\mu}}{2}-\frac{2\eps}{\pi\bar{\mu}}
\int b'(\bar{a}+\frac{x}{\bar{\mu}})\frac{x}{\bar{\mu}}|\eta(x)|^2dx,\\
\dot{\bar{\alpha}}=\frac{\eps\bar{\alpha}}{\pi\bar{\mu}}
\int b'(\bar{a}+\frac{x}{\bar{\mu}})|\eta(x)|^2dx,\\
\dot{\bar{\phi}}=-\frac{\bar{\alpha}^2\bar{\mu}^2}{4}-\frac{\eps}{\pi}
\int b(\bar{a}+\frac{x}{\bar{\mu}})|\eta(x)|^2dx-\frac{\eps}{\pi}
\int b'(\bar{a}+\frac{x}{\bar{\mu}})\frac{x}{\bar{\mu}}|\eta(x)|^2dx,\\
\dot{\bar{\mu}}=-\frac{2\eps}{\pi}\int b'(\bar{a}+\frac{x}{\bar{\mu}})|\eta(x)|^2dx,
\end{cases}
\end{align}

\noi
with the same initial data $a_0, \alpha_0, \phi_0, \mu_0$, then
\begin{align}\label{eqn a-bar{a}}
\begin{cases}
&|a-\bar{a}|\leq \tilde{c}_0\delta\eps^{\frac{1}{2}+\delta}\ln(\frac{1}{\eps}),\\
&|\alpha-\bar{\alpha}|\leq \tilde{c}_0\delta\eps^{\frac{1}{2}+\dl}
\ln(\frac{1}{\eps}),\\
&|\phi-\bar{\phi}|\leq \tilde{c}_0\delta\eps^{2\delta}\ln(\frac{1}{\eps})^2,\\
&|\mu-\bar{\mu}|\leq \tilde{c}_0\delta\eps^{\frac{1}{2}+\delta}\ln(\frac{1}{\eps}).
\end{cases}
\end{align}

\noi
where $\tilde{c}_0$ depends on $\alpha_0, \mu_0$.

As a consequence, if $\eps$ is small enough and $\frac{3}{10}<\delta<\frac{1}{2}$, then
for times \\ $0\leq t\leq \frac{\dl}{6\ln c_0}\cdot \frac{1}{\eps^{\frac{1}{2}-\delta}}\ln(\frac{1}{\eps})$
we have that
\begin{align}\label{approximation}
\|u(t)-\bar{\alpha}(t)e^{i\bar{\phi}(t)}\bar{\mu}(t)\eta(\bar{\mu}(t)(x-\bar{a}(t)))\|_{H^{\frac{1}{2}}_+}
\leq C\eps^{\frac{1}{2}+\frac{\delta}{3}}.
\end{align}

\end{theorem}

The problem of studying the solution of a perturbed equation having as initial condition
a soliton of the unperturbed equation was first addressed in the setting of the nonlinear Schr\"odinger equation by
Bronski and Jerrard in \cite{Bronki Jerrard} and their result was improved by Keraani in \cite{Keraani, Keraani II}.
They considered the semiclassical regime which is equivalent to adding a slowly varying potential $V(\eps x)$.
The method consists in using the orbital stability of the soliton and the result states that the center of mass
moves according to Newton's equation $a''(t)=-DV(a)$. It seems difficult to adapt this method to
the setting of the Szeg\"o equation since it extensively exploits the relations between the densities of mass, energy, and momentum.
These identities have no correspondent for the Szeg\"o equation.

This problem was also considered by Fr\"ohlich, Tsai, and Yau and Fr\"ohlich, \\ Gustafson, Jonsson, and Sigal
in the settings of the Hartree equation and of the nonlinear Schr\"odinger equation with
a general nonlinearity in \cite{F1,F2,F3}. Some of these results were improved in
\cite{Zworski delta, Zworski} by Zworski and Holmer in the case of the one dimensional nonlinear
Schr\"odinger equation with a Dirac potential and with a slowly varying potential.
In this paper we adapt the method of Zworski and Holmer to the case of the Szeg\"o equation.

The starting point in proving Theorem \ref{main theorem} is to determine the vector field
corresponding to the restriction $H_b|_{M}$ of the Hamiltonian to the four-dimensional manifold of solitons
\begin{align*}
M=\{ e^{i\phi}\alpha \mu \eta (x-a)),
\phi,a \in\R, \alpha>0, \mu>0 \}.
\end{align*}

\noi
Then, we determine the flow of this vector field, called the {\it effective dynamics}. In the case
of the Szeg\"o equation with a small Toeplitz potential the effective dynamics are given in the system \eqref{effective dynamics}.
We then decompose the flow of the
perturbed Szeg\"o equation \eqref{PS} into a part belonging to the manifold $M$ and a part which is symplectically orthogonal
to $M$. We show that the part of the solution which is orthogonal to $M$ is small. Thus, the flow of \eqref{PS} is close to $M$.
Then, the heuristics pointed out by
Holmer and Zworski suggest that the flow is close to the flow of $H_b|_{M}$, i.e. the effective dynamics.
This can be rigorously proved and yields the approximation \eqref{approximation}.

In proving that the part of the flow which is orthogonal to $M$ is small we consider the
Lyapunov functional and use the coerciveness of the linearized operator.

\noi
First we consider the functional
$\mathcal{E}:H^{1/2}_+\to\R$,
\begin{equation}\label{eqn E}
\mathcal{E}(u)=\frac{1}{4}\int|u|^4dx+\frac{i}{4}\int
(\partial_xu)\bar{u}dx-\frac{1}{8}\int |u|^2dx.
\end{equation}

\noi
Then $\eta=\frac{1}{x+i}$ is a critical point of $\mathcal{E}$, i.e. $d_{\eta}\mathcal{E}=0$ since
\begin{equation}\label{eqn:eta}
\frac{i}{2}\partial_x\eta+\Pi(|\eta|^2\eta)-\frac{\eta}{4}=0.
\end{equation}

\noi
The Lyapunov functional is defined by
\begin{equation*}
L(w)=\mathcal{E}(w+\eta)-\mathcal{E}(w)
\end{equation*}

\noi
and the linearized operator $\mathcal{L}:H^{\frac{1}{2}}_+\to\R$ is
\begin{equation}\label{eqn:linearized}
\mathcal{L}(w)=\mathcal{E}''_{\eta}w=
-\frac{i}{2}\partial_xw-2T_{|\eta|^2}w-H_{\eta^2}w+\frac{1}{4}w.
\end{equation}

In \cite{Zworski delta}, Holmer and Zworski
consider the case of the nonlinear cubic Schr\"odinger equation with a Dirac potential,
 that can be generalized to the case of a multiplicative linear
potential.  The maximal time for which the approximation holds is of order $\frac{1}{\sqrt{\eps}}$.
Thus, the result
we obtain for the Szeg\"o equation with a Toeplitz potential (the natural extension of the multiplicative potential)
is close to \cite{Zworski delta}.
However, working with the Lyapunov functional as it was done in \cite{Zworski delta} does not give the desired result
in the case of the Szeg\"o equation,
since we no longer have a
Galilean invariance.
Consequently, we use the linearized operator, as it was done by the above cited authors in \cite{Zworski},
in the case of a slowly varying potential.

Notice that the {\it exact effective dynamics} given by
$\bar{a},\bar{\alpha}, \bar{\phi}, \bar{\mu}$, are an approximation of the solution of the perturbed equation
only for times
\[0<t \leq \frac{\dl}{6\ln c_0}\cdot\frac{1}{\eps^{\frac{1}{2}-\delta}}\ln(\frac{1}{\eps})
\leq \frac{\dl}{6\ln c_0}\cdot\frac{1}{\eps^{\frac{1}{5}}}\ln(\frac{1}{\eps}),\]

\noi
where $\delta>\frac{3}{10}$. (If we agree to have an approximation of order $\eps^{\frac{1}{2}}$,
instead of that of order $\eps^{\frac{1}{2}+\frac{\delta}{3}}$ that we have, we can actually go up to
times $0<t \leq \frac{\dl}{6\ln c_0}\cdot\frac{1}{\eps^{\frac{1}{4}}}\ln(\frac{1}{\eps})$.)
For larger times, the approximation is only given by
$a,\alpha, \phi, \mu$, which are {\it perturbations of the effective dynamics}.
The fact that we cannot approximate the solution by the exact effective dynamics for larger times (i.e. $0<\delta<\frac{3}{10}$)
is due to the estimate on $|\phi-\bar{\phi}|$
which is only of order $O(\eps^{2\delta -})$, while we need an approximation of order $O(\eps^{\frac{1}{2}+\frac{\delta}{3}})$.
This difficulty is caused by the complicated form of the effective dynamics and by the fact that the perturbed equation does
not conserve the momentum $\|u\|_{\dot{H}^{1/2}_+}^2$. In the case of the nonlinear Schr\"odinger equation with a
Dirac or a slowly varying potential, the effective dynamics have a simpler form
and give a good approximation of the solution for all the range of times considered in \cite{Zworski delta, Zworski}.

The structure of the paper is as follows. In section 2 we briefly describe the manifold of solitons. In section 3 we find the effective dynamics.
In section 4 we use the implicit function theorem to prove the orthogonal decomposition of the flow and determine the equation of $w$, the part of the flow which is orthogonal to $M$. In section 5 we prove the coerciveness of the linearized operator in directions orthogonal to the manifold $M$. In section 6 we estimate $w$
using a bootstrap argument and in section 7 we conclude the proof of Theorem \ref{main theorem}.

\section{Manifold of solitons}

We introduce below the manifold of solitons
for the Szeg\"o equation on the real line.

For $g=(a,\alpha,\phi,\mu)\in \R\times\R_+^{\ast}\times\T\times\R_+^{\ast}$,
where $\T=\R/2\pi\Z$,
we define the following map on $L^2_+(\R)$

\begin{align*}
u\mapsto g\cdot u,\,\,\,\,\,g\cdot u (x):=e^{i\phi}\alpha \mu u(\mu(x-a)).
\end{align*}

\noi
This action gives a group structure on
$\R\times\R_+^{\ast}\times\T\times\R_+^{\ast}$:
\begin{align*}
(a,\alpha,\phi,\mu)\cdot (a',\alpha ',\phi ',\mu ')
=(a '',\alpha '',\phi '',\mu ''),
\end{align*}

\noi
where
\begin{align}\label{eqn:g 3}
\begin{cases}
a''=a+\frac{a'}{\mu}\\
\alpha ''= \alpha\alpha '\\
\phi ''=\phi+\phi '\\
\mu ''= \mu\mu '.
\end{cases}
\end{align}

\noi
We denote this group by $G$. In order to determine the Lie algebra $\mathfrak{g}$
corresponding to this Lie group,
we compute
\begin{align*}
\partial_a[(a,1,0,1)\cdot u]\Big|_{a=0}&=-\partial_x u\\
\partial_{\alpha}[(0,\alpha,0,1)\cdot u]\Big|_{\alpha=1}&=u\\
\partial_{\phi}[(0,1,\phi,1)\cdot u]\Big|_{\phi=0}&=iu\\
\partial_{\mu}[(0,1,0,\mu)\cdot u]\Big|_{\mu=1}&=x\partial_{x}u+u=\partial_x(x\cdot u).
\end{align*}

\noi
Then, the Lie algebra $\mathfrak{g}$ is generated by
\begin{align*}
e_1=-\partial_x,\,\,\,e_2=1,\,\,\,e_3=i,\,\,\,e_4= \partial_x\cdot x.
\end{align*}

\noi
It acts on $\cup_{N\in\N}\M(N)$, where
\[\M(N):=\bigg\{ \frac{A(z)}{B(z)}\in L^2_+\Big | \deg(B)=N,\, \deg(A)\leq N-1,\, B(0)=1,\,pgcd(A,B)=1\bigg\}.\]

\noi
Notice that according to \cite{Nikolskii}[Lemma 6.2.1],
we have that $\cup_{N\in\N}\M(N)$ is dense in $L^2_+(\R)$.

The action $g$ is conformally symplectic in the sense that
\begin{align}\label{eqn:confsympl}
g^{\ast}\omega=\alpha^2(g)\mu(g)\omega.
\end{align}

\noi
Indeed, with the change of variables $y=\mu(x-a)$
\begin{align*}
(g^{\ast}\omega)(u,v)&=\textup{Im}\int_{\R}e^{i\phi}\alpha
\mu u(\mu(x-a))e^{-i\phi}\alpha
\mu \bar{v}(\mu(x-a))dx\\
&=\alpha^2\mu\textup{Im}\int_{\R}u(y)\bar{v}(y)dy
=\alpha^2\mu\omega(u,v).
\end{align*}

\begin{definition}
The manifold of solitons is the orbit of $\eta$, $\eta(x)=\frac{1}{x+i}$,
under the action of the group $G$:
\begin{align*}
M=G\cdot\eta=\{ e^{i\phi}\alpha \mu \eta (\mu(x-a)),
\phi,a \in\R, \alpha>0, \mu>0 \}.
\end{align*}

\end{definition}

We then make the following identifications:
\begin{align}\label{eqn:identif}
M=G\cdot\eta\simeq G,\,\,\,T_{\eta} M=\mathfrak{g}\cdot\eta\simeq \g.
\end{align}

For $b=0$, the flow of $H_0$ is tangent to the manifold of solitons $M$.
This corresponds to the fact that if $u(0,x)\in M$, then
$u(t,x)\in M$ for all $t\in\R$.
More precisely, by equations \eqref{eqn: soliton}
and \eqref{eqn: soliton at t}, we have that if
$u(0,x)= e^{i\phi}\alpha \mu \eta (\mu (x-a))$, then
\begin{align*}
u(t,x)=g(t)\cdot\eta=e^{i\phi(t)}\alpha(t) \mu (t) \eta \big(\mu(t)(x-a(t)\big),
\end{align*}

\noi
where
\begin{align*}
\begin{cases}
\dot{a}(t)=\frac{\alpha ^2\mu}{2}\\
\dot{\alpha}(t)=0\\
\dot{\phi}(t)=-\frac{\alpha ^2\mu^2}{4}\\
\dot{\mu}(t)=0.
\end{cases}
\end{align*}

\section{Effective dynamics}

We will compute in this section the restriction
to the manifold of solitons $M$ of the
symplectic form $\omega|_{M}$
and prove that $(M,\omega|_{M})$ is a symplectic manifold.
Then, we compute the restriction of the Hamiltonian $H_b|_{M}$,
as well as the vector field associated to $H_b|_{M}$.
This vector field yields a flow on the manifold of solitons $M$,
that we refer to as the {\it effective dynamics}.

First we compute $(\omega|_{M})_{\eta}$ on $T_{\eta}M$,
at the point $\eta$. Using
\begin{align*}
\oet(e_i,e_j)=\textup{Im}
\int_{\R}(e_i\cdot\eta)(x)\cj{(e_j\cdot\eta)(x)}dx,
\end{align*}

\noi
and the residue theorem, we get
\begin{align*}
\oet(e_1,e_2)=&-\textup{Im}\int_{\R}
\partial_x(\frac{1}{x+i})\cj{\frac{1}{x+i}}dx=-\frac{\pi}{2},\\
\oet(e_1,e_3)=&0,\,\,\,\,\,\,\,\,\oet(e_1,e_4)=-\frac{\pi}{2},\,\,\,\,\,\oet(e_2,e_3)=-\pi\\
\oet(e_2,e_4)=&0,\,\,\,\,\,\,\,\oet(e_3,e_4)=\frac{\pi}{2}.
\end{align*}

\noi
Hence
\begin{align}\label{eqn:oet}
\oet=\frac{\pi}{2}(d\alpha\wedge da+d\mu\wedge da+
2d\phi\wedge d\alpha +d\phi\wedge d\mu).
\end{align}

Let us now compute $\oget$ for arbitrary $g\in G$.
By \eqref{eqn:identif} we can identify the action of $g$ on $M$
with the action $g:G\to G$ given by \eqref{eqn:g 3}.
Then, we have that the differential $d_{\eta}g:T_{\eta}M\to T_{g\cdot\eta}M$ is
given by
\begin{align}\label{eqn:dg}
d_{\eta}g=\frac{1}{\mu}da+\alpha d\alpha+d\phi+\mu d\mu.
\end{align}

\noi
By equation \eqref{eqn:confsympl}, we have that
\begin{align}\label{eqn:conformal}
\omega_{g\cdot\eta}\big(d_{\eta}g(u),d_{\eta}g(v)\big)=\alpha^2\mu\omega_{\eta}(u,v).
\end{align}

\noi
Then, equations \eqref{eqn:dg}, \eqref{eqn:conformal}, and \eqref{eqn:oet} yield
\begin{align*}
&\oget\Big(X_1(\frac{\partial}{\partial a})_{g\cdot\eta}
+X_2(\frac{\partial}{\partial \alpha})_{g\cdot\eta}
+X_3(\frac{\partial}{\partial \phi})_{g\cdot\eta}
+X_4(\frac{\partial}{\partial \mu})_{g\cdot\eta},\\
&\hphantom{XXXXX}Y_1(\frac{\partial}{\partial a})_{g\cdot\eta}
+Y_2(\frac{\partial}{\partial \alpha})_{g\cdot\eta}
+Y_3(\frac{\partial}{\partial \phi})_{g\cdot\eta}
+Y_4(\frac{\partial}{\partial \mu}\Big)_{g\cdot\eta}\Big)\\
&=\alpha^2\mu\oet\Big(\mu X_1(\frac{\partial}{\partial a})_{\eta}
+\frac{X_2}{\alpha}(\frac{\partial}{\partial \alpha})_{\eta}
+X_3(\frac{\partial}{\partial \phi})_{\eta}+\frac{X_4}{\mu}
(\frac{\partial}{\partial \mu})_{\eta},\\
&\hphantom{XXXXXXXX}\mu Y_1(\frac{\partial}{\partial a})_{\eta}
+\frac{Y_2}{\alpha}(\frac{\partial}{\partial \alpha})_{\eta}
+Y_3(\frac{\partial}{\partial \phi})_{\eta}
+\frac{Y_4}{\mu}(\frac{\partial}{\partial \mu}\Big)_{\eta}\Big)\\
&=\frac{\pi}{2}\alpha^2\mu(\frac{\mu}{\alpha}d\alpha\wedge da+d\mu\wedge da
+\frac{2}{\alpha}d\phi\wedge d\alpha+\frac{1}{\mu}d\phi\wedge d\mu)\\
&\hphantom{XXXXXXXX}\Big(X_1(\frac{\partial}{\partial a})_{\eta}
+X_2(\frac{\partial}{\partial \alpha})_{\eta}
+X_3(\frac{\partial}{\partial \phi})_{\eta}
+X_4(\frac{\partial}{\partial \mu})_{\eta},\\
&\hphantom{XXXXXXXXX}Y_1(\frac{\partial}{\partial a})_{\eta}
+Y_2(\frac{\partial}{\partial \alpha})_{\eta}
+Y_3(\frac{\partial}{\partial \phi})_{\eta}
+Y_4(\frac{\partial}{\partial \mu}\Big)_{\eta}\Big).
\end{align*}

\noi
Thus,
\begin{align}\label{eqn:oget}
\omega|_{M}=\alpha^2\mu\frac{\pi}{2}(\frac{\mu}{\alpha}d\alpha\wedge da
+d\mu\wedge da+ \frac{2}{\alpha}d\phi\wedge d\alpha +\frac{1}{\mu}d\phi\wedge d\mu).
\end{align}

\noi
One can easily verify that $\omega|_{M}$ is a non-degenerate symplectic form
and therefore, $(M,\omega|_{M})$ is a symplectic manifold.

Let $f$ be a function defined on $M\simeq G$. Then, $f$
admits a Hamiltonian vector field $X_f$ on $M$ if

\begin{align*}
\omega|_{M}(\cdot,X_f)=df=f_ada+f_{\alpha}d\alpha+f_{\mu}d\mu+f_{\phi}d\phi,
\end{align*}

\noi
where $f_a=\frac{\partial f}{\partial_a}$ and $f_{\alpha}, f_{\phi}$, and $f_{\mu}$
are defined similarly.
Denoting $X_f=X_1\frac{\partial}{\partial a}+X_2\frac{\partial}
{\partial \alpha}+X_3\frac{\partial}{\partial \phi}+X_4\frac{\partial}{\partial \mu}$
and using \eqref{eqn:oget}, the above equation is equivalent to
\begin{align*}
&\alpha^2\mu\frac{\pi}{2}\Big(\frac{\mu}{\alpha}(X_1d\alpha-X_2da)+(X_1d\mu-X_4da)
+\frac{2}{\alpha}(X_2d\phi-X_3d\alpha)+\frac{1}{\mu}(X_4d\phi-X_3d\mu)\Big)\\
&\hphantom{XXXXXXXXXXX}=f_ada+f_{\alpha}d\alpha+f_{\mu}d\mu+f_{\phi}d\phi.
\end{align*}

\noi
Then, the components of the vector field $X_f$ are
\begin{align*}
\begin{cases}
X_1=-\frac{2}{\alpha^2\mu^2\pi}(-2\mu f_{\mu}+\alpha f_{\alpha}),\\
X_2=\frac{2}{\alpha^2\mu^2\pi}(\alpha f_a+\alpha\mu f_{\phi}),\\
X_3=\frac{2}{\alpha^2\mu\pi}(\mu f_{\mu}-\alpha f_{\alpha}),\\
X_4=-\frac{2}{\alpha^2\mu\pi}(\mu f_{\phi}+2f_{a}).
\end{cases}
\end{align*}

\noi
This allows us to determine the Hamiltonian flow associated to $X_f$, $\dot{u}=X_f(u)$,
which is given by $(\dot{a},\dot{\alpha},\dot{\phi},\dot{\mu})=(X_1,X_2,X_3,X_4)$.

Let us now compute $H_b|_{M}$ and find its Hamiltonian vector field.
\begin{align*}
H_b|_{M}(g\cdot\eta)&=\frac{1}{4}\int_{\R}\alpha^4\mu^4|\eta(\mu(x-a))|^4dx
+\frac{\eps}{2}\int_{\R}b(x)\alpha^2\mu^2|\eta(\mu(x-a))|^2dx\\
&=\frac{\alpha^4\mu^3}{4}\int_{\R}|\eta(x)|^4dx
+\frac{\eps\alpha^2\mu^2}{2}\int_{\R}b(x)|\eta(\mu(x-a))|^2dx\\
&=\frac{\alpha^4\mu^3\pi}{8}+\frac{\eps\alpha^2\mu}{2}
\int_{\R}b\big(a+\frac{x}{\mu}\big)|\eta(x)|^2dx.
\end{align*}

\noi
Taking $f=H_b|_{M}$, we have that
\begin{align*}
\begin{cases}
f_a=\frac{\eps\alpha^2\mu}{2}\int b'(a+\frac{x}{\mu})|\eta(x)|^2dx,\\
f_{\alpha}=\frac{\pi\alpha^3\mu^3}{2}+\eps\alpha\mu
\int b(a+\frac{x}{\mu})|\eta(x)|^2dx,\\
f_{\phi}=0,\\
f_{\mu}=\frac{3\pi\alpha^4\mu^2}{8}+\frac{\eps\alpha^2}{2}
\int b(a+\frac{x}{\mu})|\eta(x)|^2dx-\frac{\eps\alpha^2}{2}
\int b'(a+\frac{x}{\mu})\frac{x}{\mu}|\eta(x)|^2dx.
\end{cases}
\end{align*}

\noi
As above, we determine the components of the Hamiltonian vector field
associated to $f=H_b|_{M}$, and obtain that the flow of $H_b|_{M}$ is given by
\begin{align*}
\begin{cases}
\dot{a}=\frac{\alpha^2\mu}{2}-\frac{2\eps}{\pi\mu}\int b'(a+\frac{x}{\mu})
\frac{x}{\mu}|\eta(x)|^2dx,\\
\dot{\alpha}=\frac{\eps\alpha}{\pi\mu}\int b'(a+\frac{x}{\mu})|\eta(x)|^2dx,\\
\dot{\phi}=-\frac{\alpha^2\mu^2}{4}-\frac{\eps}{\pi}\int b(a+\frac{x}{\mu})
|\eta(x)|^2dx-\frac{\eps}{\pi}\int b'(a+\frac{x}{\mu})\frac{x}{\mu}|\eta(x)|^2dx,\\
\dot{\mu}=-\frac{2\eps}{\pi}\int b'(a+\frac{x}{\mu})|\eta(x)|^2dx.
\end{cases}
\end{align*}

\section{Reparametrized evolution}

Our goal is to show that the flow generated by $H_b$ can be approximated by the
effective flow of $H_b|_{M}$. In order to do so, we decompose the solution $u(t)$
of the Szeg\"o equation with small Toeplitz potential \eqref{PS}, into a component belonging
to $M$ and a component which is symplectically orthogonal to $M$ in the sense that:
\begin{equation}\label{eqn:decomposition}
u(t)=g(t)\cdot (\eta+w(t)),\,\,\,\,\omega (w(t),X\eta)=0, \forall X\in\g.
\end{equation}

\noi
The key point is to prove that the orthogonal component $w$ is small.

Let us show that the above decomposition/reparametrization is indeed possible at least for short time.

\begin{lemma}\label{lemma: def w}
For a compact subset $\Sigma$ of $\R\times\R^{\ast}_+\times \T\times\R^{\ast}_+$
and $\gamma>0$, denote by
\begin{equation*}
U_{\Sigma,\gamma}=\Big\{u\in H^{\frac{1}{2}_+}; \inf_{g\in\Sigma}\|u-g\cdot\eta\|_{H^{\frac{1}{2}}_+}<\gamma\Big\}.
\end{equation*}

\noi
a $\gamma$-tubular neighborhood of $\Sigma$.

There exists $\gamma_0=\gamma_0(\Sigma)$ such that if $u\in U_{\Sigma,\gamma}$, with $\gamma\leq\gamma_0$,
then there exists a unique element $g(u)\in\Sigma$ with the property
\begin{equation*}
\omega(g(u)^{-1}\cdot u-\eta,X\cdot\eta)=0, \forall X\in\g.
\end{equation*}

\end{lemma}

\begin{proof}
Consider the function $F:H^{\frac{1}{2}}_+\times G\to\g^{\ast}$,
\begin{equation*}
F(u,h)(X)=\omega (h\cdot u-\eta,X\cdot\eta).
\end{equation*}

\noi
We want to solve $F(u,h)=0$ for $h=h(u)$.
We verify that the function $F$ satisfies the hypotheses of the
Implicit Function Theorem:\\
(i) $F(u,h)$ is of class $C^1$ in $h$,\\
(ii) $F(g\cdot\eta,g^{-1})=0$ for all $g\in G$,\\
(iii) $d_hF(g\cdot \eta, g^{-1}):T_{g^{-1}}G\to\g^{\ast}$ is invertible
for all $g\in G$.

The first two properties can be checked directly.
As for the third property, it is enough to check it for $g=e=(1,0,1,0)$,
the unity of the group $G$. Thus, since $T_eG=\g$, it is enough to check that
$d_hF(\eta,e):\g\to\g^{\ast}$ is invertible. But \\ $d_hF(\eta,e)=(\omega|_{M})_{\eta}$
which is non-degenerate because, in the basis $\{e_j\cdot\eta\}_{j=1}^4$
of $\g$, it writes
\begin{align*}
\frac{\pi}{2}\left(
\begin{matrix}
0 & -1 & 0 & -1\\
1 & 0 & -2 & 0\\
0 & 2 & 0 & 1\\
1 & 0 & -1 & 0
\end{matrix}
\right),
\end{align*}

\noi
whose determinant does not vanish.
\end{proof}

Thus, the orthogonal decomposition \eqref{eqn:decomposition}, with
$w(t)=g(t)^{-1}\cdot u(t)-\eta$, holds as long as $u(t)$ is close enough to $M=G\cdot\eta$.

In order to find the equation that $w$ satisfies, we need
the following lemmas:

\begin{lemma}\label{lemma: gY}
If $t\mapsto g(t)=(a(t),\alpha (t),\phi(t),\mu(t))$ is a $C^1$
function and \\$u\in \cup_{N\in\N}\M(N)$, then
\begin{align*}
\frac{d}{dt}g(t)\cdot u=g(t)\cdot(Y(t)u),
\end{align*}

where $Y(t)=\dot{a}(t)\mu(t) e_1+\frac{\dot{\alpha}(t)}{\alpha (t)}e_2
+\dot{\phi}(t)e_3+\frac{\dot{\mu}(t)}{\mu (t)}e_4$.
\end{lemma}

\begin{proof}
\begin{align*}
\frac{d}{dt}g(t)\cdot u=&\frac{d}{dt}(e^{i\phi}\alpha\mu u(\mu (x-a)))\\
=&i\dot{\phi}e^{i\phi}\alpha\mu u(\mu(x-a))
+e^{i\phi}\dot{\alpha}\mu u(\mu(x-a))+e^{i\phi}\alpha\dot{\mu} u(\mu (x-a))\\
&+e^{i\phi}\alpha\mu \partial _xu(\mu (x-a))\dot{\mu}x
-e^{i\phi}\alpha\mu \partial_x u(\mu (x-a))(\dot{\mu}a+\mu\dot{a})\\
=&\dot{\phi}g\cdot (e_3\cdot u)+\frac{\dot{\alpha}}{\alpha}g\cdot (e_2\cdot u)
+\frac{\dot{\mu}}{\mu}g\cdot (e_4\cdot u)+\dot{a}\mu g\cdot (e_1\cdot u)\\
=&g\cdot (Y(t)u).
\end{align*}
\end{proof}

We also need Lemma 2.1 from \cite{Zworski}, that we restate in the context of our problem.

\begin{lemma}\label{lemma:Xgast}
Suppose that $g:H^{\frac{1}{2}}_+\to H^{\frac{1}{2}}_+$ is a diffeomorphism such
that \\ $g^{\ast}\omega=\rho(g)\omega$, where
$\rho(g)\in C^{\infty}(H^{\frac{1}{2}}_+,\R^{\ast})$.
Then, for $f\in C^{\infty}(H^{\frac{1}{2}}_+,\R)$ we have that
\begin{align*}
(g^{-1})_{\ast}X_f(g(\rho))=\frac{1}{\rho(g)}X_{g^{\ast}f}(\rho),
\rho\in H^{\frac{1}{2}}_+.
\end{align*}

\end{lemma}

In the next proposition we determine the equation satisfied by $w$.
\begin{proposition}\label{Prop eqn w}
If the solution of the perturbed Szeg\"o equation \eqref{Szego T} can be
reparametrized as in Lemma \ref{lemma: def w}, $u(t)=g(t)\cdot (\eta+w(t))$, for all $t$ in an interval
$(t_1,t_2)$, then $w$ satisfies the following equation:

\begin{align*}
\partial_t w&=-X\eta+\Big(-i\eps\Pi\big(b(a+\frac{x}{\mu})\eta\big)
+2B e_1\cdot\eta-C e_2\cdot\eta
+(A+B)e_3\cdot\eta+2C e_4\cdot\eta\Big)\\
&-Xw+\Big(-i\eps\Pi\big(b(a+\frac{x}{\mu})w\big)
+2B e_1\cdot w-C e_2\cdot w
+(A+B)e_3\cdot w+2C e_4\cdot w\Big)\\
&+i\alpha^2\mu^2\mathcal{L}w-i\alpha^2\mu^2\mathcal{N}w,
\end{align*}

\noi
where
\begin{align}\label{eqn:X}
&X:=\big(\dot{a}\mu-\frac{\alpha^2\mu^2}{2}+2B\big)e_1
+\big(\frac{\dot{\alpha}}{\alpha}-C\big)e_2+\big(\dot{\phi}
+\frac{\alpha^2\mu^2}{4}+A+B\big)e_3\\
&\hphantom{XXXXXX}+
\big(\frac{\dot{\mu}}{\mu}+2C\big)e_4,\notag\\
&\mathcal{L}w:=-\frac{i}{2}\partial_xw-2T_{|\eta|^2}w-H_{\eta^2}w+\frac{1}{4}w,\notag\\
&\Nn w:=\Pi\big(|w|^2w+|w|^2\eta+2w\textup{Re}(\eta\bar{w})\big),\notag\\
&A:=\frac{\eps}{\pi}\int b(a+\frac{x}{\mu})|\eta(x)|^2dx,\notag\\
&B:=\frac{\eps}{\pi}\int b'(a+\frac{x}{\mu})x|\eta(x)|^2\frac{dx}{\mu},\notag\\
&C:=\frac{\eps}{\pi}\int b'(a+\frac{x}{\mu})|\eta(x)|^2\frac{dx}{\mu}.\notag
\end{align}

\end{proposition}

\begin{proof}
Denote $\tilde{u}=w+\eta=g^{-1}u$. Then, by Lemma \ref{lemma: gY}, we have that
\begin{align*}
\partial_t u=\partial_t(g\cdot(\eta+w))=g\cdot Y(\eta+w)+g\cdot \partial_t w.
\end{align*}

\noi
Then, Lemma \ref{lemma:Xgast} yields
\begin{align*}
\partial_t w&=-Y(\eta+w)+g^{-1} \partial_t u=-Y(\eta+w)+g^{-1}X_{H_{b}}(u)
=-Y(\eta+w)+g^{-1}X_{H_{b}}(g\tilde{u})\\
&=-Y(\eta+w)+\frac{1}{\alpha^2g}X_{g^{\ast}H_b}(\tilde{u}).
\end{align*}

\noi
Since
\begin{align*}
(g^{\ast}H_b)(\tilde{u})=H_b(g\tilde{u})=\frac{\alpha^4\mu^3}{4}\int |\tilde{u}|^4dx
+\frac{\eps\alpha^2\mu}{2}\int b\big(a+\frac{x}{\mu}\big)|\tilde{u}|^2dx,
\end{align*}

\noi
we have that
\begin{align*}
(X_{g^{\ast}H_b})(\tilde{u})=-i\Pi \Big(\alpha^4\mu^3|\tilde{u}|^2\tilde{u}
+\eps\alpha^2\mu b(a+\frac{x}{\mu})\tilde{u}\Big)
\end{align*}

\noi
and therefore,
\begin{align}\label{eqn w inter}
\partial_t w=&-Y(\eta+w)-\frac{i}{\alpha^2\mu}\Pi\Big(\alpha^4\mu^3|\eta+w|^2(\eta+w)
+\eps\alpha^2\mu b\big(a+\frac{x}{\mu}\big)(\eta+w)\Big)\\
=&\Big(-Y\eta-i\eps\Pi(b(a+\frac{x}{\mu})\eta)\Big)+\Big(-Yw-i\eps\Pi(b(a+\frac{x}{\mu})w)\Big)\notag\\
&-i\alpha^2\mu^2\Pi\Big(2\textup{Re}(\eta\bar{w})\eta+|\eta|^2w\Big)\notag\\
&-i\alpha^2\mu^2\Pi\Big(|w|^2w+2\textup{Re}(\eta\bar{w})w+|w|^2\eta\Big)
-i\alpha^2\mu^2\Pi(|\eta|^2\eta)\notag.
\end{align}

\noi
Denoting
\begin{align}\label{some eqn for X}
X=Y+\big(-\frac{\alpha^2\mu^2}{2}+2B\big)e_1-Ce_2
+\big(\frac{\alpha^2\mu^2}{4}+A+B\big)e_3+2Ce_4,
\end{align}

\noi
and noticing that
\begin{align*}
-\frac{\alpha^2\mu^2}{2}e_1\cdot\eta=\frac{\alpha^2\mu^2}{2}\partial_x\eta,
\,\,\,\,\frac{\alpha^4\mu^3}{4}e_3\cdot\eta=\frac{i}{4}\alpha^4\mu^3\eta,
\end{align*}

\noi
and similar relations hold for $w$,
we obtain
\begin{align*}
\partial_t w=&-X\eta+\Big(-i\eps\Pi\big(b(a+\frac{x}{\mu})\eta\big)
+2Be_1\cdot\eta-C e_2\cdot \eta
+(A+B)e_3\cdot\eta+2C e_4\cdot\eta \Big)\\
&-Xw+\Big(-i\eps\Pi\big(b(a+\frac{x}{\mu})w\big)+2Be_1\cdot w-C e_2\cdot w
+(A+B)e_3\cdot w+2C e_4\cdot w\Big)\\
&-i\alpha^2\mu^2\Big(\Pi\big(2|\eta|^2w+\eta^2\bar{w}\big)+\frac{i}{2}\partial_x w
-\frac{w}{4}\Big)-i\alpha^2\mu^2\Big(\Pi(|\eta|^2\eta)
+\frac{i}{2}\partial_x \eta-\frac{\eta}{4}\Big)\\
&-i\alpha^2\mu^2\Pi\Big(|w|^2w+2\textup{Re}(\eta\bar{w})w+|w|^2\eta\Big).
\end{align*}

\noi
Equation \eqref{eqn:eta} and \eqref{eqn:linearized} yield the conclusion.
\end{proof}

\begin{remark}\label{Rk}
\rm
Notice that
$X\equiv 0$ is equivalent to $a,\alpha,\phi,\mu$
satisfying the effective dynamics \eqref{effective dynamics}.
\end{remark}

\begin{lemma}\label{lemma:alpha2mu}
If the solution of the perturbed Szeg\"o equation \eqref{Szego T} can be
reparametrized as in Lemma \ref{lemma: def w}, $u(t)=g(t)\cdot (\eta+w(t))$
at time $t$, then the $L^2$-norm of $w(t)$ is equal to
\begin{align*}
\|w(t)\|_{L^2}^2=\pi\Big(\frac{\alpha_0^2\mu_0}{\alpha^2(t)\mu(t)}-1\Big).
\end{align*}

\noi
Consequently,
$\alpha^2(t)\mu(t)\leq \alpha^2_0\mu_0$.
\end{lemma}

\begin{proof}
By the conservation of the $L^2$-norm of the solution of the
Szeg\"o equation with a Toeplitz potential, we have that
\begin{align*}
\|\eta+w(t)\|_{L^2}^2=\|g(t)^{-1}u(t)\|_{L^2}^2=\frac{1}{\alpha^2(t)\mu(t)}\|u(t)\|_{L^2}^2
=\frac{\|u(0)\|_{L^2}^2}{\alpha^2(t)\mu(t)}=\frac{\pi\alpha^2_0\mu_0}{\alpha^2(t)\mu(t)}.
\end{align*}

\noi
By the orthogonality of $w$ and $\eta$, we have that $\omega (w,X\cdot\eta)=0$, for all $X\in\g$.
In particular, taking $X=e_3$, we obtain
\begin{align*}
\jb{w,\eta}=\textup{Re}\int w\cj{\eta}dx=-\textup{Im}\int w\cj{i\eta}dx=-\omega(w,e_3\cdot\eta)=0.
\end{align*}

\noi
Then
\begin{align*}
\|\eta+w(t)\|_{L^2}^2=\|\eta\|_{L^2}^2+\|w(t)\|_{L^2}^2=\pi+\|w(t)\|_{L^2}^2,
\end{align*}

\noi
and the conclusion follows.
\end{proof}

Next we define $P$, the symplectically orthogonal projection
on the manifold of solitons $M$. We also give two technical lemmas
concerning some properties of $P$.
\begin{definition}
Define the projection onto $T_{\eta}M=\g\cdot\eta\simeq \g$ by \\$P:\big(\cup_{N\in\N}\M(N)\big)'\to\g$,
\begin{align*}
\omega(u-P(u)\eta,Y\eta)=0, \forall Y\in\g.
\end{align*}
\end{definition}

\begin{lemma}\label{lemma:estimatesP}
Let $\|\cdot\|$ be a norm on $\g$ obtained by using the standard $\R^4$ norm
in the basis $\{e_1,e_2,e_3,e_4\}$. Then, for all $w\in H^{\frac{1}{2}}_+$ and $Y\in\g$, we have
\begin{align*}
&\|P(Yw)\|\leq C\|Y\| \|w\|_{L^2},\\
&\|P(i\Nn w)\|\leq C \|w\|^2_{H^{\frac{1}{2}}_+}(\|w\|_{H^{\frac{1}{2}}_+}+1).
\end{align*}

\end{lemma}

\begin{proof}
Let $P=\sum_{j=1}^4P_je_j$, $P_j:H^{-\frac{1}{2}}_+\to\R$.
Then the definition of $P$ yields
\begin{align*}
\oet(u-\sum_{j=1}^4P_je_j\cdot\eta, a_1e_1\cdot\eta
+a_2e_2\cdot\eta+a_3e_3\cdot\eta+a_4e_4\cdot\eta)=0,
\end{align*}

\noi
for all $a_i\in\R$. Then, it follows that
\begin{align*}
&a_1\big(\omega(u,e_1\cdot\eta)-\frac{\pi}{2}P_2-\frac{\pi}{2}P_4\big)
+a_2\big(\omega(u,e_2\cdot\eta)+\frac{\pi}{2}P_1-\pi P_3\big)\\
&+a_3\big(\omega(u,e_3\cdot\eta)+\pi P_2+\frac{\pi}{2}P_4\big)
+a_4\big(\omega(u,e_4\cdot\eta)+\frac{\pi}{2}P_1-\frac{\pi}{2}P_3\big)=0,
\end{align*}

for all $a_i\in\R$. Therefore,
\begin{align*}
\begin{cases}
P_1(u)=\frac{2}{\pi}\Big(\omega(u,e_2\cdot\eta)-2\omega(u,e_4\cdot\eta)\Big),\\
P_2(u)=\frac{2}{\pi}\Big(-\omega(u,e_3\cdot\eta)-\omega(u,e_1\cdot\eta)\Big),\\
P_3(u)=\frac{2}{\pi}\Big(\omega(u,e_2\cdot\eta)-\omega(u,e_4\cdot\eta)\Big),\\
P_4(u)=\frac{2}{\pi}\Big(2\omega(u,e_1\cdot\eta)+\omega(u,e_3\cdot\eta)\Big).
\end{cases}
\end{align*}

\noi
The conclusion follows by using the Cauchy-Schwarz inequality
 and integration by parts. For example, for $P_1$ we have
\begin{align*}
\|P_1(Yw)\|\leq &\Big|\int Yw\bar{\eta}\Big|
+2\Big|\int Yw\cj{\partial_x(x\eta)}\Big|\\
=&\Big|\int \big(-Y_1\partial_x w+Y_2w+iY_3w
+Y_4\partial_x(xw)\big)\bar{\eta}dx\Big|\\
&+2\Big|\int \big(-Y_1\partial_x w+Y_2w+iY_3w
+Y_4\partial_x(xw)\big)\cj{\partial_x(x\eta)}dx\Big|\\
\leq &\|Y\|\Big(\Big|\int w\partial_x\bar{\eta}dx\Big|
+2\Big|\int w\bar{\eta}dx\Big|
+\Big|\int xw\partial_x\bar{\eta}dx\Big|
+2\Big|\int w\partial_x^2(x\bar{\eta})dx\Big|\\
&+4\Big|\int w\partial_x(x\bar{\eta})dx\Big|
+2\Big|\int xw\partial^2_x(x\bar{\eta})dx\Big|\Big)\\
\leq & C\|Y\|\|w\|_{L^2}\big(\|\partial_x\eta\|_{L^2}
+\|\eta\|_{L^2}+\|x\partial_x\eta\|_{L^2}+\|\partial^2_x(x\eta)\|_{L^2}
+\|x\partial^2_x(x\eta)\|_{L^2}\big)\\
\leq &C\|Y\|\|w\|_{L^2}.
\end{align*}

\noi
By using the Sobolev embedding $H^{\frac{1}{2}}(\R)\subset L^p(\R)$
for all $2\leq p<\infty$, we have
\begin{align*}
\|P_1(i\Nn w)\|=&\|\omega(i\Nn w,\eta)-2\omega(i\Nn w,\partial_x(x\eta))\|\\
\leq &\Big|\int |w|^2w\bar{\eta}dx+\int |w|^2|\eta|^2dx
+2\int w\textup{Re}(\eta\bar{w})\bar{\eta}dx\Big|\\
&+2\Big| \int|w|^2w\partial_x(x\bar{\eta})dx+\int |w|^2\eta\partial_x(x\bar{\eta})dx
+2\int w\textup{Re}(\eta\bar{w})\partial_x(x\bar{\eta})dx\Big|\\
\leq &C(\|w^2\|_{L^2}+\|w^3\|_{L^2})\leq C\|w\|_{L^4}(\|w\|_{L^4}+\|w\|_{L^8}^2)
\leq \|w\|_{H^{\frac{1}{2}}_+}^2(\|w\|_{H^{\frac{1}{2}}_+}+1).
\end{align*}

\end{proof}

\begin{lemma}
If $f:\R\to\R$ is a function of class $C^1$
such that $f'\in L^1(\R)\cap L^2(\R)$ and $f\in L^{\infty}(\R)$, then
\begin{align*}
P(\Pi(if\eta))&=\frac{2}{\pi}\Big(\int f'(x)x|\eta(x)|^2dx\Big)
e_1-\frac{1}{\pi} \Big(\int f'(x)|\eta(x)|^2dx\Big) e_2\\
&+\frac{1}{\pi}\Big(\int f(x)|\eta(x)|^2dx+\int f'(x)x|\eta(x)|^2dx\Big)e_3
+\frac{2}{\pi}\Big(\int f'(x)|\eta(x)|^2dx\Big) e_4.
\end{align*}

\end{lemma}

\begin{proof}
Let $Y=\sum_{j=1}^4a_je_j$ be an arbitrary vector in $\g$.
Then, integrating by parts we have
\begin{align*}
\omega(\Pi(if&\eta),Y\cdot\eta)=\omega (if\eta,a_1e_1\cdot \eta
+a_2e_2\cdot\eta+a_3e_3\cdot\eta+a_4e_4\cdot\eta)\\
&=\textup{Im}\Big(-a_1\int if\eta\partial_x\bar{\eta}dx
+a_2\int if\eta\bar{\eta}dx\\
&\hphantom{XXXX}+a_3\int if\eta (-i)\bar{\eta}dx+a_4\int
if\eta\partial_x(x\bar{\eta})dx\Big)\\
&=-\frac{a_1}{2}\int f\partial_x(|\eta|^2)dx+a_2\int f|\eta|^2dx
+a_4\textup{Re}\int f(x)\eta(x)\big(\bar{\eta}(x)+x\partial_x\bar{\eta}(x)\big)dx\\
&=\frac{a_1}{2}\int f'|\eta|^2dx+(a_2+a_4)\int f|\eta|^2dx
-\frac{a_4}{2}\int\big(xf'(x)+f(x)\big)|\eta(x)|^2dx\\
&=\frac{a_1}{2}\int f'|\eta|^2dx+(a_2+\frac{a_4}{2})
\int f|\eta|^2dx-\frac{a_4}{2}\int f'(x)x|\eta(x)|^2dx.
\end{align*}

\noi
Using the formula for $\oet$ we have
\begin{align*}
&\omega\Big(\frac{2}{\pi}\Big(\int f'(x)x|\eta(x)|^2dx\Big) e_1\cdot\eta
-\frac{1}{\pi} \Big(\int f'(x)|\eta(x)|^2dx\Big) e_2\cdot\eta\\
&+\frac{1}{\pi}\Big(\int f(x)|\eta(x)|^2dx+\int f'(x)x|\eta(x)|^2dx\Big)e_3\cdot\eta+\frac{2}{\pi}\Big(\int f'(x)|\eta(x)|^2dx\Big) e_4\cdot\eta,Y\cdot\eta\Big)\\
&=\frac{a_1}{2}\int f'|\eta|^2dx+(a_2+\frac{a_4}{2})
\int f|\eta|^2dx-\frac{a_4}{2}\int f'(x)x|\eta(x)|^2dx.
\end{align*}

\noi
By the definition of the projection $P$, the conclusion follows.
\end{proof}

\begin{lemma}\label{lemma:P(X)}
\begin{align*}
P\Big(-i\eps\Pi\big(b(a+\frac{x}{\mu})\eta\big)
+2B e_1\cdot\eta-C e_2\cdot\eta
+(A+B)e_3\cdot\eta+2C e_4\cdot\eta\Big)=0.
\end{align*}
\end{lemma}

\begin{proof}
Take $f(x)=\eps b(a+\frac{x}{\mu})$ in the above lemma.
\end{proof}

\begin{remark}
\rm
Lemma \ref{lemma:P(X)}
and equation \eqref{some eqn for X} show that
\[P\Big(-Y\eta-i\eps\Pi\big(b(a+\frac{x}{\mu})\eta\big)\Big)
=-X-\frac{\alpha^2\mu^2}{2}e_1+\frac{\alpha^2\mu^2}{4}e_3.\]

\noi
Thus, $X$ is the orthogonal projection on the manifold of solitons
of a significant term of the right-hand side of the equation \eqref{eqn w inter} satisfied by $w$.
\end{remark}

In the following we intend to give an estimate for $\|X\|$.
We need the following definition and Lemma that we cite from \cite[Lemma 2.2]{Zworski}.

Let $f\in C^{\infty}(H^{\frac{1}{2}}_+,\R)$ and suppose $df(\rho_0)=0$.
Then the Hessian of $f$ at $\rho_0$ is well defined
$f''(\rho_0):T_{\rho_0}H^{\frac{1}{2}}_+\to T^{\ast}_{\rho_0}H^{\frac{1}{2}}_+$.
 We identify $T_{\rho_0}H^{\frac{1}{2}}_+$ and $T^{\ast}_{\rho_0}H^{\frac{1}{2}}_+$
 using the inner product and we define the {\it Hamiltonian map}
 $F:T_{\rho_0}H^{\frac{1}{2}}_+\to T_{\rho_0}H^{\frac{1}{2}}_+$ by
\begin{align*}
F=-if''(\rho_0),\,\,\,\,\ \jb{f''(\rho_0)X,Y}=\omega(Y,FX).
\end{align*}

\begin{lemma}\label{lemma:Hamiltonian map}
Let $N\subset H^{\frac{1}{2}}_+$ be a finite-dimensional
symplectic submanifold of $H^{\frac{1}{2}}_+$ and let
$f\in C^{\infty}(H^{\frac{1}{2}}_+,\R)$ such that
\begin{align*}
X_f(\rho)\in T_{\rho}N\subset T_{\rho} H^{\frac{1}{2}}_+, \rho\in N.
\end{align*}

If $\rho_0\in N$ and $df(\rho_0)=0$, then the Hamiltonian map satisfies
\begin{align*}
F(T_{\rho}N)\subset T_{\rho}N.
\end{align*}

\end{lemma}

\begin{lemma}\label{lemma:X}
If the solution of the perturbed Szeg\"o equation \eqref{Szego T} can be
reparametrized as in Lemma \ref{lemma: def w}, $u(t)=g(t)\cdot (\eta+w(t))$, for all $t$ in an interval
$(t_1,t_2)$, $\|w(t)\|_{L^2}$ is small enough, and $\frac{\mu_0}{2}\leq \mu(t)\leq \frac{3\mu_0}{2}$, then
the vector $X$ defined by
\begin{align*}
&X=\big(\dot{a}\mu-\frac{\alpha^2\mu^2}{2}+2B\big)e_1
+\big(\frac{\dot{\alpha}}{\alpha}-C\big)e_2+\big(\dot{\phi}
+\frac{\alpha^2\mu^2}{4}+A+B\big)e_3+
\big(\frac{\dot{\mu}}{\mu}+2C\big)e_4,
\end{align*}

\noi
where the expressions of $A,B,C$
can be found in equation \eqref{eqn:X}, satisfies the inequality
\begin{align*}
\|X\|\leq C(\eps\|w\|_{L^2}+\|w\|_{H^{\frac{1}{2}}_+}^2
+\|w\|_{H^{\frac{1}{2}}_+}^3).
\end{align*}
\end{lemma}

\begin{remark}
\rm
Lemma \ref{lemma:X} yields that if $\|w\|_{H^{1/2}_+}$ is small,
then $\|X\|$ is also small. On the other hand,
we noticed in Remark \ref{Rk} that $\|X\|$ measures
how far $a,\alpha,\phi,\mu$ are from the effective dynamics
\eqref{effective dynamics}. Thus, the Lemma \ref{lemma:X} shows that if one
can prove that $w$, the part of the flow which is
orthogonal to the manifold of solitons, is small,
then $a,\alpha,\phi,\mu$ are perturbations of the effective dynamics.
\end{remark}
\begin{proof}
Note first that $P(Y\cdot\eta)=Y$, for all $Y\in\g$.

Since $\omega(w,Y\cdot\eta)=0$, for all $Y\in\g$, it
follows that $Pw=0$ and $P\partial_tw=\partial_tPw=0$.
Then, by Proposition \ref{Prop eqn w} and Lemma \ref{lemma:P(X)}, we have
\begin{align*}
0=&-X-P(Xw)+\alpha^2\mu^2P(i\Ll w)-\alpha^2\mu^2P(i\Nn w)\\
&+P\Big(-i\eps\Pi(b(a+\frac{x}{\mu})w)+2B e_1\cdot w
-C e_2\cdot w+(A+B)e_3\cdot w+2C e_4\cdot w\Big).
\end{align*}

\noi
By Lemma \ref{lemma:estimatesP}, we have that
\begin{align*}
&\|P(Xw)\|\leq c\|X\| \|w\|_{L^2},\\
&\|P(i\Nn w)\|\leq c \|w\|^2_{H^{\frac{1}{2}}_+}(\|w\|_{H^{\frac{1}{2}}_+}+1).
\end{align*}

\noi
We prove that $P(-i\Ll w)=0$. For $\mathcal{E}$ defined by
equation  \eqref{eqn E}, we have that $X_{\mathcal{E}}$
is tangent to $M$, which corresponds to the fact that
 if the initial data is in  $M$, then the flow of $H_0$ stays in $M$. Then,
\begin{align*}
(X_{\mathcal{E}})_{g\cdot\eta}\subset T_{g\cdot\eta}M\subset T_{g\cdot\eta}
 H^{\frac{1}{2}}_+.
\end{align*}

\noi
Then, by Lemma \ref{lemma:Hamiltonian map}, we have that the Hamiltonian
map of $\mathcal{E}$, $-i\Ll$, satisfies
\begin{align*}
(-i\Ll )(T_{\eta}M)\subset T_{\eta} M.
\end{align*}

\noi
Then, since $w$ is orthogonal to $T_{\eta}M=\g\cdot\eta$ and $T_{|\eta|^2}$,
$H_{\eta^2}$ are symmetric with respect to the real scalar product, we obtain that
\begin{align*}
\omega (-i\Ll w,X\cdot\eta)&=\textup{Im}\int -i\Ll w \cj{X\cdot\eta}dx
=-\textup{Re}\int \Ll w\cj{X\cdot\eta}dx=-\jb{\Ll w,X\cdot\eta}\\
&=-\jb{w,\Ll(X\cdot\eta)}
=\textup{Im} \int w\cj{-i\Ll (X\cdot\eta)}dx=\omega \big(w,(-i\Ll )(X\cdot\eta)\big)=0.
\end{align*}
\noi
For the last term, we first notice that we have
\begin{align}\label{estimate ABC}
&|A|=\frac{\eps}{\pi}\Big|\int b(a+\frac{x}{\mu})|\eta(x)|^2dx\Big|
\leq c\eps \|b\|_{L^{\infty}}\|\eta\|_{L^2}^2\leq c\eps , \\
&|B|=\frac{\eps}{\pi}\Big|\int b'(a+\frac{x}{\mu})x|\eta(x)|^2\frac{dx}{\mu}\Big|
\leq c\eps \|b'\|_{L^{1}}\|x\eta^{2}(x)\|_{L^{\infty}}\leq c\eps,\notag\\
&|C|=\frac{\eps}{\pi}\Big|\int b'(a+\frac{x}{\mu})|\eta(x)|^2\frac{dx}{\mu}\Big|
\leq c\eps\|b'\|_{L^{1}}\|\eta\|_{L^{\infty}}^2\leq c\eps.\notag
\end{align}

\noi
Using the expression of $P$ we found in the proof of Lemma
\ref{lemma:estimatesP}, we obtain that
\begin{align*}
&\Big\|P\Big(-i\eps\Pi\big(b(a+\frac{x}{\mu})w\big)
+2B e_1\cdot w-C e_2\cdot w
+(A+B)e_3\cdot w+2C e_4\cdot w\Big)\Big\|\\
&\leq c\eps\|w\|_{L^2}.
\end{align*}

\noi
By  Lemma \ref{lemma:alpha2mu} we have that $\alpha^2\mu\leq\alpha^2_0\mu_0$, and thus we have
\begin{align*}
\|X\|&\leq c(\|X\|\|w\|_{L^2}+\mu\|w\|_{H^{\frac{1}{2}}_+}^2
+\mu\|w\|_{H^{\frac{1}{2}}_+}^3)+c\eps\|w\|_{L^2}.
\end{align*}

\noi
If $\|w\|_{L^2}$ is small enough so that $c\|w\|_{L^2}<1$, then we write
\begin{align*}
(1-c\|w\|_{L^2})\|X\|\leq c(\eps\|w\|_{L^2}+\mu\|w\|_{H^{\frac{1}{2}}_+}^2
+\mu\|w\|_{H^{\frac{1}{2}}_+}^3),
\end{align*}

\noi
To conclude, we use the fact that $\mu(t)\leq\frac{3\mu_0}{2}$.
\end{proof}

\section{Coerciveness of the linearized operator $\Ll$}

In this section we prove that the linearized operator $\Ll$,
defined by equation \eqref{eqn:linearized}, is coercive in directions which are
symplectically orthogonal to the manifold of solitons $M$.
\begin{lemma}\label{lemma:Ker}
For all $f\in\textup{Ker}(H_{\eta^2})\cap H^{\frac{1}{2}}_+$, we have that
\begin{align*}
\jb{\Ll(f),f}\geq \frac{1}{4}\|f\|_{H^{\frac{1}{2}}_+}^2.
\end{align*}
\end{lemma}

\begin{proof}
Since $\eta(x)=\frac{1}{x+i}$, we have that
$\textup{Ker}(H_{\eta^2})=\Big(\frac{x-i}{x+i}\Big)^2L^2_+$.
Let \\ $f\in \textup{Ker}(H_{\eta^2})\cap H^{\frac{1}{2}}_+$,
$f=\Big(\frac{x-i}{x+i}\Big)^2h$, where $h\in H^{\frac{1}{2}}_+$.
Then
\begin{align*}
T_{|\eta|^2}f=\Pi\Big(\frac{1}{(x+i)(x-i)}(\frac{x-i}{x+i})^2h\Big)
=\Pi\Big(\frac{x-i}{(x+i)^3}h\Big)=\frac{x-i}{(x+i)^3}h
\end{align*}

\noi
and
\begin{align*}
\Ll(f)&=-\frac{i}{2}\partial_xf-2T_{|\eta|^2}f-H_{\eta^2}f+\frac{1}{4}f\\
&=2\frac{x-i}{(x+i)^3}h-\frac{i}{2}\Big(\frac{x-i}{x+i}\Big)^2\partial_xh
-2\frac{x-i}{(x+i)^3}h+\frac{1}{4}\Big(\frac{x-i}{x+i}\Big)^2h\\
&=\Big(\frac{x-i}{x+i}\Big)^2(-\frac{i}{2}\partial_xh+\frac{1}{4}h),
\end{align*}

\noi
and thus, using $|\frac{x-i}{x+i}|=1$ and the Plancherel identity, we obtain
\begin{align*}
\jb{\Ll(f),f}&=\jb{\Big(\frac{x-i}{x+i}\Big)^2(-\frac{i}{2}\partial_xh
+\frac{1}{4}h),\Big(\frac{x-i}{x+i}\Big)^2h}
=\jb{-\frac{i}{2}\partial_xh+\frac{1}{4}h,h}\\
&=\frac{1}{2}\int_0^{\infty}\xi|\hat{h}(\xi)|d\xi
+\frac{1}{4}\|f\|_{L^2}^2\geq \frac{1}{4}\|f\|_{H^{\frac{1}{2}}_+}^2.
\end{align*}
\end{proof}

In what follows we need a Kronecker-type theorem
characterizing the Hankel operators of finite rank. We state this theorem bellow. For the
proof we refer to \cite{pocov1}.
\begin{theorem}[\cite{pocov1}]
The Hankel operator $H_u$ has finite rank $N$ if and only if $u$  is a rational function
which belongs to $\M(N)$, where
\[\M(N)=\bigg\{ \frac{A(z)}{B(z)}\in L^2_+\Big | \deg(B)=N,\, \deg(A)\leq N-1,\, B(0)=1,\,p.g.c.d.(A,B)=1\bigg\}.\]

Moreover, if $u\in \M(N)$, $u(z)=\frac{A(z)}{B(z)}$, where $B(z)=\prod_{j=1}^{J}(z-p_j)^{m_j}$,
 with $\sum_{j=1}^Jm_j=N$ and $\textup{Im} p_j<0$ for all $j=1,2,...,J$,
 then the range of $H_u$ is given by
\begin{equation*}
\textup{Ran}\,H_u=\textup{span}_{\C}\bigg\{\frac{1}{(z-p_j)^{m}}, 1\leq m\leq m_j\bigg\}_{j=1}^J
\end{equation*}

\end{theorem}

\begin{proposition}\label{prop:coercivity}
If $w\in H^{\frac{1}{2}}_+$ is such that
$\omega(w,X\cdot\eta)=0$, for all $X\in\g$, then
\begin{align*}
\jb{\Ll w,w}\geq \frac{1}{4}\|w\|_{H^{\frac{1}{2}}_+}^2.
\end{align*}
\end{proposition}

\begin{proof}
By the Kronecker-type theorem, we have that the range $\textup{Ran}\,H_{\eta^2}$
is generated by all the fractions having as a numerator a complex number and as a denominator a factor of
$\eta^2$. More precisely,
\begin{align*}
\textup{Ran}(H_{\eta^2})&=\textup{span}_{\R}\Big\{\frac{1}{x+i},\frac{i}{x+i},
\frac{1}{(x+i)^2},\frac{i}{(x+i)^2}\Big\}\\
&=\textup{span}_{\R}\{\eta,i\eta,-i\partial_x\eta,i\partial_x(x\eta)\}\\
&=\textup{span}_{\R}\{ie_1\cdot\eta, ie_2\cdot\eta, ie_3\cdot\eta, ie_4\cdot\eta \}.
\end{align*}

\noi
On the other hand, we have that $\omega(w,X\cdot \eta)=0$ for all $X\in\g$, which is equivalent to
\[0=\omega(w,e_j\cdot \eta)=\textup{Im}\int w\cj{e_j\cdot\eta}dx=\textup{Re}\int w\cj{ie_j\cdot\eta}dx
=\jb{w,ie_j\cdot\eta},\]

\noi
for $j=1,2,3,4$.
Thus $w$ belongs to the orthogonal of $\textup{Ran}(H_{\eta^2})$
with respect to the real scalar product. Since $H_{\eta^2}$ is $\C$-antilinear,
$w$ belongs also to the orthogonal with respect to the Hermitian inner product
in $L^2$, which is $\textup{Ker}(H_{\eta^2})$. Hence
\\ $w\in\textup{Ker}(H_{\eta^2})\cap H^{\frac{1}{2}}_+$. By Lemma
\ref{lemma:Ker}, the conclusion then follows.
\end{proof}

\section{Main estimates}

In this section we estimate $w$, the part of the flow
which is symplectically orthogonal to the manifold of solitons,
and prove that it is small.
\begin{lemma}\label{main estimate}
If the solution of the perturbed Szeg\"o equation \eqref{Szego T} can be
reparametrized as in Lemma \ref{lemma: def w}, $u(t)=g(t)\cdot (\eta+w(t))$
on a time interval $(t_1,t_2)$, $\frac{\mu_0}{2}\leq\mu(t)\leq\frac{\mu_0}{2}$,
and $w(t)$ is small enough in the $H^{\frac{1}{2}}_+$-norm,
 then the following estimate holds
\begin{align*}
\frac{1}{2}\big|\partial_t\jb{\Ll w,w}\big|\leq c\eps\|w\|_{H^{\frac{1}{2}}_+}
+c\eps\|w\|_{H^{\frac{1}{2}}_+}^2
+c\|w\|_{H^{\frac{1}{2}}_+}^3+c\|w\|_{H^{\frac{1}{2}}_+}^5,
\end{align*}

\noi
where $c$ is a constant depending on $\alpha_0$ and $\mu_0$.
\end{lemma}

\begin{proof}
We have that
\begin{align*}
\frac{1}{2}\partial_t\jb{\Ll w&,w}=\jb{\Ll w,\partial_t w}\\
=&\jb{\Ll w, -X\eta}\\
&+\jb{\Ll w, -i\eps\Pi\big(b(a+\frac{x}{\mu})\eta\big)
+2B e_1\cdot\eta-C e_2\cdot\eta
+(A+B)e_3\cdot\eta+2C e_4\cdot\eta}\\
&+\jb{\Ll w, -X w}\\
&+\jb{\Ll w, -i\eps\Pi\big(b(a+\frac{x}{\mu})w\big)
+2B e_1\cdot w-C e_2\cdot w
+(A+B)e_3\cdot w+2C e_4\cdot w)}\\
&+\jb{\Ll w,i\alpha^2\mu^2\mathcal{L}w}-
\jb{\Ll w,i\alpha^2\mu^2\mathcal{N}w}\\
=&\mathrm{I+II+III+IV+V+VI}
\end{align*}

\noi
and we will estimate each of the six terms.
The challenge is to deal with the terms containing $\partial_xw$
since we only have $w\in H^{\frac{1}{2}}_+(\R)$. In what follows we focus
on such terms, the rest of the terms being easier to handle.

We set $X=\sum_{j=1}^4a_je_j$. By Lemma \ref{lemma:X}, we have that
\begin{align*}
|a_j|\leq c(\eps\|w\|_{L^2}+\|w\|_{H^{\frac{1}{2}}_+}^2+\|w\|_{H^{\frac{1}{2}}_+}^3).
\end{align*}

\noi
For $\mathrm{I}$, we integrate by parts
\begin{align*}
\textup{Re}\Big(\bar{a}_1\int \frac{i}{2}\partial_x w\cj{e_1\cdot\eta}dx\Big)
=\textup{Re}\Big(\bar{a}_1\int \frac{i}{2} w \partial_x^2\bar{\eta} dx\Big)
\end{align*}

\noi
and apply the Cauchy-Schwarz inequality for each term. We obtain
\begin{align*}
|\mathrm{I}|\leq c\|X\|\|w\|_{L^2}\leq c(\eps\|w\|^2_{L^2}+\|w\|_{H^{\frac{1}{2}}_+}^3
+\|w\|_{H^{\frac{1}{2}}_+}^4).
\end{align*}

\noi
For II, integrating by parts and using Cauchy-Schwarz, we have
\begin{align*}
\textup{Re}\Big(\int\frac{i}{2}\partial_x w\cdot i\cj{\eps b(a+\frac{x}{\mu})}\bar{\eta}dx\Big)&
=\frac{\eps}{2}\textup{Re}\Big(\int w \cj{b'(a+\frac{x}{\mu})}\frac{1}{\mu}\bar{\eta}dx
+\int w \cj{b(a+\frac{x}{\mu})}\bar{\eta}'dx\Big)\\
&\leq c\eps\|w\|_{L^2}\|\eta\|_{L^{\infty}}\|b'\|_{L^{2}}\frac{\mu^{1/2}}{\mu}
+c\eps\|w\|_{L^2}\|\eta'\|_{L^2}\|b\|_{L^{\infty}}\\
&\leq c\eps (1+\frac{1}{\mu^{1/2}})\|w\|_{L^2}.
\end{align*}

\noi
Using the equation \eqref{estimate ABC} for the rest of the terms, we obtain
\begin{align*}
|\mathrm{II}|\leq c\eps \|w\|_{L^2}.
\end{align*}

\noi
For III and IV we analyze each term. Besides integrating by parts and using \\
Cauchy-Schwarz or H\"older inequalities, a key ingredient is the fact that
 we deal with the real scalar product.
\begin{align*}
\mathrm{III}&=\jb{\Ll w,-X w}=\textup{Re}\Big(\sum_{j=1}^4a_j\int\frac{i}{2}\partial_x w\cj{e_j\cdot w}dx
+2\sum_{j=1}^4a_j\int|\eta|^2w\cj{e_j\cdot w}dx\\
&+\sum_{j=1}^4a_j\int\eta^2\cj{w}\cj{e_j\cdot w}dx
-\frac{1}{4}\sum_{j=1}^4a_j\int w\cj{e_j\cdot w}dx\Big)=\mathrm{(i)+(ii)+(iii)+(iv)}.
\end{align*}

\noi
Then
\begin{align*}
\mathrm{(i)}&=\textup{Re}\Big(-a_1\frac{i}{2}\int |\partial_x w|^2
+a_2\frac{i}{2}\int\partial_x w\cj{w}dx+\frac{a_3}{4}\int\partial_x(|w|^2)dx\\
&\hphantom{XXXX}+a_4\frac{i}{2}\int\partial_x w\cj{w}dx
+a_4\frac{i}{2}\int x|\partial_x w|^2dx\Big)\\
&=-\frac{a_2+a_4}{2}\int\frac{1}{i}\partial_x w\cj{w}dx
=-\frac{a_2+a_4}{2}\int_0^{\infty}\xi|\hat{w}(\xi)|^2d\xi\\
&=-\frac{a_2+a_4}{2}\|w\|_{\dot{H}^{1/2}_+}^2\leq \|X\|\|w\|_{H^{\frac{1}{2}}_+}^2,
\end{align*}

\noi
by the H\"older inequality we have
\begin{align*}
\mathrm{(ii)}&=2\textup{Re}\Big(-a_1\int |\eta|^2w\partial_x\cj{w}dx
+a_2\int |\eta|^2|w|^2dx-a_3i\int |\eta|^2|w|^2dx\\
&\hphantom{XXXX}+a_4\int |\eta|^2|w|^2dx+a_4\int|\eta|^2xw\partial_x\cj{w}dx\Big)\\
&=-a_1\int |\eta|^2\partial_x(|w|^2)dx+2(a_2+a_4)\int |\eta|^2|w|^2dx
+a_4\int |\eta|^2x\partial_x(|w|^2)dx\\
&\leq c\|X\|\|w\|_{L^2}^2
\end{align*}

\noi
similarly
\begin{align*}
\mathrm{(iii)}&=\textup{Re}\Big(-\frac{a_1}{2}\int \eta^2\partial_x(\cj{w}^2)dx
+a_2\int \eta^2\cj{w}^2dx-a_3i\int \eta^2\cj{w}^2dx\\
&\hphantom{XXXX}+a_4\int \eta^2\cj{w}^2dx+\frac{a_4}{2}\int\eta^2x\partial_x(\cj{w}^2)dx\Big)\\
&\leq c\|X\|\|w\|_{L^2}^2,
\end{align*}

\noi
and
\begin{align*}
\mathrm{(iv)}&=-\frac{1}{4}\textup{Re}\Big(-a_1\int w\partial_x\cj{w}dx+a_2\int |w|^2dx-a_3i\int |w|^2dx\\
&\hphantom{XXXX}+a_4\int |w|^2dx+\frac{a_4}{2}\int xw\partial_x\cj{w}dx\Big)\\
&=-\frac{1}{4}\Big(-\frac{a_1}{2}\int \partial_x(|w|^2)dx+(a_2+a_4)\int |w|^2dx+\frac{a_4}{2}\int x\partial_x(|w|^2)dx\Big)\\
&=-\frac{1}{4}(a_2+\frac{a_4}{2})\int |w|^2dx
\leq c\|X\|\|w\|_{L^2}^2.
\end{align*}

\noi
Hence
\begin{align*}
|\mathrm{III}|\leq c\|X\|\|w\|_{H^{\frac{1}{2}}_+}^2\leq c(\eps\|w\|^3_{H^{\frac{1}{2}}_+}
+\|w\|_{H^{\frac{1}{2}}_+}^4+\|w\|_{H^{\frac{1}{2}}_+}^5).
\end{align*}

\noi
For IV we have
\begin{align*}
&\mathrm{IV}=\textup{Re}\int \frac{i}{2}\partial_x w \Big(i\eps b(a+\frac{x}{\mu})\bar{w}
+2B\cj{e_1\cdot w}-C\cj{e_2\cdot w}
+(A+B)\cj{e_3\cdot w}+2C \cj{e_4\cdot w}\Big)dx\\
&+2\textup{Re}\int |\eta|^2w \Big(i\eps\cj{\Pi \big( b(a+\frac{x}{\mu})w\big)}
+2B\cj{e_1\cdot w}-C \cj{e_2\cdot w}
+(A+B)\cj{e_3\cdot w}+2C\cj{e_4\cdot w}\Big)dx\\
&+\textup{Re}\int \eta^2\bar{w} \Big(i\eps \cj{\Pi \big( b(a+\frac{x}{\mu})w\big)}
+2B\cj{e_1\cdot w}-C\cj{e_2\cdot w}
+(A+B)\cj{e_3\cdot w}+2C\cj{e_4\cdot w}\Big)dx\\
&-\frac{1}{4}\textup{Re}\int w \Big(i\eps b(a+\frac{x}{\mu})\bar{w}
+2B\cj{e_1\cdot w}-C\cj{e_2\cdot w}
+(A+B)\cj{e_3\cdot w}+2C\cj{e_4\cdot w}\Big)dx\\
&=\textrm{(i)+(ii)+(iii)+(iv)}.
\end{align*}

\noi
By the equations \eqref{estimate ABC} and the Sobolev embedding
$H^{\frac{1}{2}}(\R)\subset L^{p}(\R)$, $2\leq p<\infty$, we have
\begin{align*}
\textrm{(i)}=&-\frac{1}{4}\int \eps b(a+\frac{x}{\mu})\partial_x(|w|^2)dx
-B\textup{Re}\int i |\partial_xw|^2dx
-\frac{C}{2}\textup{Re}\int i\partial_xw\cj{w}dx\\
&+\frac{A+B}{4}\int\partial_x(|w|^2)dx+C\textup{Re}
\int i \partial_x w\cj{w}dx\\
=&\frac{\eps}{4\mu}\int b'(a+\frac{x}{\mu})|w|^2dx
-\frac{C}{2}\|w\|_{\dot{H}^{\frac{1}{2}}_+}^2\leq \frac{c\eps}{\mu^{1/2}}\|b'\|_{L^2}\|w\|_{L^4}^2
+c\eps\|w\|_{\dot{H}^{\frac{1}{2}}_+}^2\\
\leq &c\eps\big(1+\frac{1}{\mu^{1/2}}\big)\|w\|_{H^{\frac{1}{2}}_+}^2,
\end{align*}

\noi
For $\mathrm{(ii)}$ we only analyze the terms containing $\partial_x w$.
By the equations \eqref{estimate ABC}, we obtain
\begin{align*}
&-4B\textup{Re}\int |\eta|^2w\partial_x\cj{w}dx
+4C\textup{Re}\int |\eta|^2xw\partial_x\cj{w}dx\\
&=-2B\int |\eta|^2\partial_x(|w|^2)dx
+2C\int |\eta|^2x\partial_x(|w|^2)dx\\
&=2B\int \partial_x(|\eta|^2)|w|^2dx
-2C\int \partial_x(|\eta|^2x)|w|^2dx\\
&\leq c\eps\|w\|_{L^2}^2.
\end{align*}

\noi
Thus
\begin{align*}
\mathrm{(ii)}\leq c\eps(1+\frac{1}{\mu^{1/2}})\|w\|_{L^2}^2
\end{align*}

\noi
and similarly we obtain the same bound for $\mathrm{(iii)}$.
Computing the last term, we obtain that (iv)=0. Hence
\begin{align*}
|\mathrm{IV}|\leq c\eps(1+\frac{1}{\mu^{1/2}})\|w\|_{H^{\frac{1}{2}}_+}^2.
\end{align*}

\noi
Since we work with the real scalar product,
it follows immediately that V=0. For VI again we only analyze the terms containing
$\partial_xw$. The important step is to group together $w\bar{\eta}+\bar{w}\eta\in\R$.
\begin{align*}
&-\alpha^2\mu^2\jb{\frac{i}{2}\partial_x w, i(|w|^2w+2|w|^2\eta+w^2\cj{\eta})}\\
&=-\alpha^2\mu^2\Big(\frac{1}{4}\int|w|^2\partial_x(|w|^2))dx+\frac{1}{2}\textup{Re}\int|w|^2\partial_xw\cj{\eta}dx
+\frac{1}{2}\textup{Re}\int\partial_xw\cj{w}(w\bar{\eta}+\cj{w}\eta)dx\Big)\\
&=-\alpha^2\mu^2\Big(\frac{1}{8}\int\partial_x(|w|^4)dx
+\frac{1}{2}\textup{Re}\int|w|^2\partial_xw\cj{\eta}dx
+\frac{1}{2}\int\partial_x(|w|^2)2\textup{Re}(w\cj{\eta})dx\Big)\\
&=-\alpha^2\mu^2\Big(\frac{1}{2}\textup{Re}\int|w|^2\partial_xw\cj{\eta}dx
-\frac{1}{2}\int|w|^22\textup{Re}(\cj{\eta}\partial_x w
+w\partial_x\cj{\eta})dx\Big)\\
&=\alpha^2\mu^2\textup{Re}\int |w|^2w\partial_x\cj{\eta}dx
\leq c\alpha^2\mu^2\|w\|_{H^{\frac{1}{2}}_+}^3.
\end{align*}

\noi
For the other terms it is enough to apply the Cauchy-Schwarz
and Sobolev inequalities. Using Lemma \ref{lemma:alpha2mu} we obtain
\begin{align*}
|\mathrm{VI}|\leq c\alpha^2\mu^2(\|w\|_{H^{\frac{1}{2}}_+}^3+\|w\|_{H^{\frac{1}{2}}_+}^4)\leq c\alpha_0^2\mu_0\mu(\|w\|_{H^{\frac{1}{2}}_+}^3+\|w\|_{H^{\frac{1}{2}}_+}^4).
\end{align*}
\end{proof}

In the following, we combine the inequality in Lemma \ref{main estimate}
with the coerciveness properties of the linearized operator $\Ll$,
to obtain an estimate for $\|w\|_{H^{\frac{1}{2}}_+}$.
\begin{proposition}\label{prop:bootstrat step}
Suppose the solution of the perturbed Szeg\"o equation \eqref{Szego T} can be
reparametrized as in Lemma \ref{lemma: def w}, $u(t)=g(t)\cdot (\eta+w(t))$
on a time interval $[t_1,t_2]$ and $\frac{\mu_0}{2}\leq\mu(t)\leq \frac{\mu_0}{2}$.
Let $0<\eps\ll 1$ and $0<\dl<\frac{1}{2}$. If $|t_2-t_1|\leq \frac{1}{\eps^{\frac{1}{2}-\dl}}$ and
\begin{align*}
\|w\|_{L^{\infty}([t_1,t_2],H^{\frac{1}{2}}_+)}\leq \eps^{\frac{1}{2}},
\end{align*}

\noi
then
\begin{align*}
\|w\|_{L^{\infty}([t_1,t_2],H^{\frac{1}{2}}_+)}
\leq c_0\|w(t_1)\|_{H^{\frac{1}{2}}_+}+c_0\eps^{\frac{1+\delta}{2}},
\end{align*}

\noi
where $c_0>2$ is a constant depending only on $\alpha_0$ and $\mu_0$.
\end{proposition}

\begin{proof}
Integrating from $t_1$ to $t_2$ the estimate
in Lemma \ref{main estimate}, we have that
\begin{align*}
\jb{\Ll w(t_2),w(t_2)}\leq & \jb{\Ll w(t_1),w(t_1)}
+c(t_2-t_1)\eps\|w\|_{L^{\infty}([t_1,t_2],H^{\frac{1}{2}}_+)}\\
&+c(t_2-t_1)\eps\|w\|_{L^{\infty}([t_1,t_2],H^{\frac{1}{2}}_+)}^2
+c(t_2-t_1)\|w\|_{L^{\infty}([t_1,t_2],H^{\frac{1}{2}}_+)}^3\\
&+c(t_2-t_1)\|w\|_{L^{\infty}([t_1,t_2],H^{\frac{1}{2}}_+)}^4
+(t_2-t_1)\|w\|_{L^{\infty}([t_1,t_2],H^{\frac{1}{2}}_+)}^5.
\end{align*}

\noi
On the other hand, we have
\begin{align*}
\jb{\Ll w(t_1),w(t_1)}=&\frac{1}{2}\textup{Re}\int\frac{1}{i}\partial_xw(t_1)\cj{w}(t_1)dx
-2\int |\eta|^2|w(t_1)|^2dx\\
&-\textup{Re}\int \eta^2\cj{w}(t_1)^2dx+\frac{1}{4}\int|w(t_1)|^2dx\\
\leq &\frac{1}{2}\|w(t_1)\|_{H^{\frac{1}{2}}_+}^2+2\|\eta\|_{L^{\infty}}^2\|w(t_1)\|_{L^2}^2
+\|\eta\|_{L^{\infty}}^2\|w(t_1)\|_{L^2}^2+\frac{1}{4}\|w(t_1)\|_{L^2}^2\\
\leq & 4 \|w(t_1)\|_{H^{\frac{1}{2}}_+}^2.
\end{align*}

\noi
Together with the coerciveness of the linearized operator $\Ll$
in Proposition \ref{prop:coercivity}, this yields
\begin{align*}
\frac{1}{4}\|w\|_{L^{\infty}([t_1,t_2],H^{\frac{1}{2}}_+)}^2\leq &4 \|w(t_1)\|_{H^{\frac{1}{2}}_+}^2
+c(t_2-t_1)\eps\|w\|_{L^{\infty}([t_1,t_2],H^{\frac{1}{2}}_+)}\\
&+c(t_2-t_1)\eps\|w\|_{L^{\infty}([t_1,t_2],H^{\frac{1}{2}}_+)}^2+c(t_2-t_1)\|w\|_{L^{\infty}([t_1,t_2],H^{\frac{1}{2}}_+)}^3\\
&+c(t_2-t_1)\|w\|_{L^{\infty}([t_1,t_2],H^{\frac{1}{2}}_+)}^5.
\end{align*}

\noi
Since $c(t_2-t_1)\eps=c\eps^{\frac{1}{2}+\dl}<\frac{1}{8}$ we can pass the term
$c(t_2-t_1)\eps\|w\|_{H^{\frac{1}{2}}_+}^2$ to the left hand-side of the inequality and
with the estimates in the hypothesis we obtain
\begin{align*}
\frac{1}{8}\|w\|_{L^{\infty}([t_1,t_2],H^{\frac{1}{2}}_+)}^2&\leq 4 \|w(t_1)\|_{H^{\frac{1}{2}}_+}^2+3c\eps^{1+\dl}.
\end{align*}

\noi
This gives us the conclusion with the constant $c_0=\max(32,24c)$
depending only on $\alpha_0,\mu_0$.
\end{proof}

The proposition below is the main step
in proving Theorem \ref{main theorem}.
\begin{proposition}\label{prop: bootstrap estimate}
Let $\Sigma$ be a compact subset of $\R\times\R^{\ast}_+\times\T\times\R^{\ast}_+$, $0<\delta<\frac{1}{2}$,
and let $\eps>0$ be such that $\eps^{\frac{1}{2}}<\gamma_0$,
where $\gamma_0$ was defined in Lemma \ref{lemma: def w}.
Suppose $\inf_{g\in\Sigma}\|u(0)-g\cdot\eta\|_{H^{\frac{1}{2}}_+}\leq\eps^{\frac{1}{2}+\frac{\dl}{2}}$.
Then, for all
\[0<t\leq\frac{\dl}{6\ln c_0}\cdot\frac{1}{\eps^{\frac{1}{2}-\dl}}
\ln(\frac{1}{\eps}),\]

\noi
the solution of the perturbed Szeg\"o equation \eqref{Szego T} at time $t$
can be parameterized as in Lemma \ref{lemma: def w},
$u(t)=g(t)(\eta+w(t))$. Moreover, we have
\begin{align}\label{central estim}
\|w\|_{L^{\infty}([0,t],H^{\frac{1}{2}}_+)}
\leq &\eps^{-\frac{\dl}{6}}\|w(0)\|_{H^{\frac{1}{2}}_+}
+\eps^{\frac{1}{2}+\frac{\delta}{3}}\\
\intertext{and}
\frac{\mu_0}{2}\leq&\mu(t)\leq \frac{3\mu_0}{2}.\notag
\end{align}

\end{proposition}

\begin{proof}
We use a bootstrap argument. Set
\begin{equation}\label{def T}
T:=\sup\big\{t>0\Big| \inf_{g\in\Sigma}\|u(t)-g\cdot\eta\|_{L^{\infty}([0,t],H^{1/2}_+)}\leq\eps^{\frac{1}{2}},
\frac{\mu_0}{2}\leq\mu(t)\leq \frac{3\mu_0}{2}\Big\}.
\end{equation}

\noi
We intend to show that $T\geq \frac{\dl}{6\ln c_0}\cdot\frac{1}{\eps^{\frac{1}{2}-\dl}}
\ln(\frac{1}{\eps})$.
Suppose by contradiction that
$$T< \frac{\dl}{6\ln c_0}\cdot\frac{1}{\eps^{\frac{1}{2}-\dl}}
\ln(\frac{1}{\eps}).$$

\noi
Since $\inf_{g\in\Sigma}\|u(t)-g(t)\cdot\eta\|_{L^{\infty}([0,t],H^{1/2}_+)}\leq\eps^{\frac{1}{2}}<\gamma_0$
for all $0<t<T$, it follows by Lemma \ref{lemma: def w} that the solution of the perturbed
Szeg\"o equation \eqref{Szego T} can be reparametrized as $u(t)=g(t)\cdot(\eta+w(t))$ for all $0<t<T$, and moreover
$\|w(t)\|_{L^{\infty}([0,t],H^{1/2}_+)}\leq\eps^{\frac{1}{2}}$.
We apply the Proposition \ref{prop:bootstrat step} successively
on the intervals $[0,\frac{1}{\eps^{\frac{1}{2}-\dl}}]$,
$[\frac{1}{\eps^{\frac{1}{2}-\dl}},\frac{2}{\eps^{\frac{1}{2}-\dl}}]$,...,
$[\frac{k-1}{\eps^{\frac{1}{2}-\dl}},\frac{k}{\eps^{\frac{1}{2}-\dl}}]$.
For $t$ in the interval $[0,\frac{1}{\eps^{\frac{1}{2}-\dl}}]$,
we obtain
\begin{align*}
\|w(t)\|_{H^{\frac{1}{2}}_+}\leq c_0\|w(0)\|_{H^{\frac{1}{2}}_+}
+c_0\eps^{\frac{1+\dl}{2}}.
\end{align*}

\noi
Using this information for $t=\frac{1}{\eps^{\frac{1}{2}-\dl}}$,
we  obtain for $t\in [\frac{1}{\eps^{\frac{1}{2}-\dl}},\frac{2}{\eps^{\frac{1}{2}-\dl}}]$ that
\begin{align*}
\|w(t)\|_{H^{\frac{1}{2}}_+}\leq c_0^2\|w(0)\|_{H^{\frac{1}{2}}_+}
+c_0(1+c_0)\eps^{\frac{1+\dl}{2}}.
\end{align*}

\noi
Ultimately, we have that for all $t\in [0,\frac{k}{\eps^{\frac{1}{2}-\dl}}]$
\begin{align*}
\|w(t)\|_{H^{\frac{1}{2}}_+}&\leq c_0^k\|w(0)\|_{H^{\frac{1}{2}}_+}
+c_0(\sum_{j=0}^{k-1}c_0^j)\eps^{\frac{1+\dl}{2}}
=c_0^k\|w(0)\|_{H^{\frac{1}{2}}_+}+c_0
\frac{c_0^{k}-1}{c_0-1}\eps^{\frac{1+\dl}{2}}.
\end{align*}

\noi
Since $c_0>2$, we have that $c_0
\frac{c_0^{k}-1}{c_0-1}\leq 2c_0^k$. Take $k$ such that $c_0^{k}=\eps^{-\frac{\delta}{6}}$,
which is equivalent to
\[k=\frac{\dl}{6\ln c_0}\ln(\frac{1}{\eps}).\]

\noi
Then,
\begin{align*}
\|w\|_{L^{\infty}([0,\frac{k}{\eps^{\frac{1}{2}-\dl}}],H^{\frac{1}{2}}_+)}
\leq \eps^{-\frac{\delta}{6}}\|w(0)\|_{H^{\frac{1}{2}}_+}+2\eps^{\frac{1}{2}+\frac{\dl}{3}}\leq 3\eps^{\frac{1}{2}+\frac{\dl}{3}}.
\end{align*}

\noi
Therefore, we have
for $0\leq t\leq \frac{\dl}{6\ln c_0}\cdot\frac{1}{\eps^{\frac{1}{2}-\delta}}
\ln(\frac{1}{\eps})$ that
\begin{align}\label{contrad bootstrap}
&\|w(t)\|_{L^{\infty}([0,t],H^{\frac{1}{2}}_+)}
\leq 3\eps^{\frac{1}{2}+\frac{\dl}{3}},\\
\intertext{and by Lemma \ref{lemma:X} it follows that}
&\|X\|\leq c\eps^{1+\frac{2\dl}{3}}.\notag
\end{align}

\noi
By the definition of $X$ \eqref{eqn:X}, it follows that
\[\Big|\frac{\dot{\mu}}{\mu}+2C\Big|\leq c\eps^{1+\frac{2\delta}{3}}.\]

\noi
Thus
\[\frac{\dot{\mu}}{\mu}\leq-\frac{2\eps}{\pi}\int b'(a+\frac{x}{\mu})|\eta(x)|^2\frac{dx}{\mu}
+c\eps^{1+\frac{2\delta}{3}}.\]

\noi
Integrating from $0$ to $t$, where $0\leq t\leq \frac{\dl}{6\ln c_0}\cdot\frac{1}{\eps^{\frac{1}{2}-\delta}}
\ln(\frac{1}{\eps})$, and using the change of variables $y=a+\frac{x}{\mu}$, we obtain that
\[\ln\big(\frac{\mu(t)}{\mu_0}\big)\leq c(\eps\|b'\|_{L^1}\|\eta\|_{L^{\infty}}^2+\eps^{1+\frac{2\delta}{3}})t
\leq \frac{c\dl}{6\ln c_0}\cdot \eps^{\frac{1}{2}+\delta}\ln(\frac{1}{\eps}).\]

\noi
Since around zero we have the Taylor expansion $\ln(1+x)=x+O(x^2)$, it follows that
\[\frac{\mu(t)-\mu_0}{\mu_0}
\leq \frac{c\dl}{6\ln c_0}\cdot \eps^{\frac{1}{2}+\delta}\ln(\frac{1}{\eps}).\]

\noi
Hence, we obtain
\[|\mu(t)-\mu_0|
\leq \tilde{c}_0\dl\eps^{\frac{1}{2}+\delta}\ln(\frac{1}{\eps}),\]

\noi
where $\tilde{c}_0$ is a constant depending on $\alpha_0,\mu_0$.
Thus,
\begin{equation}\label{contrad bootstrap 2}
\frac{2\mu_0}{3}\leq\mu(t)\leq\frac{4\mu_0}{3}\,\,\,\,\,\,\,\,\, \text{ for }\,\,\,\,\,\,\,\,\,\,
0\leq t\leq \frac{\dl}{6\ln c_0}\cdot\frac{1}{\eps^{\frac{1}{2}-\delta}}\ln(\frac{1}{\eps}).
\end{equation}

Equations \eqref{contrad bootstrap} and \eqref{contrad bootstrap 2}
show that the conditions in the definition of $T$ \eqref{def T} hold with better bounds,
i.e. $3\eps^{\frac{1}{2}+\frac{\delta}{3}}$ instead of $\eps^{\frac{1}{2}}$,
$\frac{4\mu_0}{3}$ instead of $\frac{3\mu_0}{2}$,
and $\frac{2\mu_0}{3}$ instead of $\frac{\mu_0}{2}$, for
$0\leq t\leq \frac{\dl}{6\ln c_0}\cdot\frac{1}{\eps^{\frac{1}{2}-\delta}}
\ln(\frac{1}{\eps})$.
Since $w(t)$ and $\mu(t)$ are continuous with respect to time,
it follows that there exists $t_0>0$ such that
the conditions in the definition of $T$ with exactly the same bounds as in that definition
hold for times $0\leq t\leq \frac{\dl}{6\ln c_0}\cdot\frac{1}{\eps^{\frac{1}{2}-\delta}}
\ln(\frac{1}{\eps})+t_0$.
This contradicts our assumption $T< \frac{\dl}{6\ln c_0}\cdot\frac{1}{\eps^{\frac{1}{2}-\dl}}
\ln(\frac{1}{\eps})$.
Therefore, the conclusion of the proposition follows.
\end{proof}

\section{Proof of Theorem \ref{main theorem}}

In this section we prove that Theorem \ref{main theorem} follows from Proposition \ref{prop: bootstrap estimate}.
\begin{proof}[Proof of Theorem \ref{main theorem}]
First we notice that $u(0)=g(0)\cdot\eta$, where \\ $g(0)=(a_0,\alpha_0,\phi_0,\mu_0)$.
Thus, by Proposition \ref{prop: bootstrap estimate}, it
follows that $u(t)$ can be reparametrized as $u(t)=g(t)\cdot(\eta+w(t))$
for times $0\leq t\leq \frac{\dl}{6\ln c_0}\cdot\frac{1}{\eps^{\frac{1}{2}-\delta}}
\ln(\frac{1}{\eps})$,
and moreover
\begin{align*}
&\|w(t)\|_{L^{\infty}([0,t],H^{\frac{1}{2}}_+)}
\leq 3\eps^{\frac{1}{2}+\frac{\dl}{3}},\,\,\,\,\,\,\frac{\mu_0}{2}\leq \mu(t)\leq \frac{3\mu_0}{2}.
\end{align*}

\noi
By Lemma \ref{lemma:X}
we then obtain
\begin{equation}\label{estimate X}
\|X\|\leq c\eps^{1+\frac{2\dl}{3}}.
\end{equation}

\noi
Proceeding as in the last part of the proof of Proposition \ref{prop: bootstrap estimate},
we obtain that
\begin{align*}
|\mu(t)-\mu_0|\leq \tilde{c}_0\dl\eps^{\frac{1}{2}+\delta}
\ln(\frac{1}{\eps}).
\end{align*}

\noi
for $0\leq t\leq \frac{\dl}{6\ln c_0}\cdot\frac{1}{\eps^{\frac{1}{2}-\delta}}
\ln(\frac{1}{\eps})$. Similarly we have
$|\bar{\mu}(t)-\mu_0|\leq \tilde{c}_0\dl\eps^{\frac{1}{2}+\delta}
\ln(\frac{1}{\eps})$. Then
\begin{align*}
|\mu(t)-\cj{\mu}(t)|\leq \tilde{c}_0\dl\eps^{\frac{1}{2}+\delta}
\ln(\frac{1}{\eps}).
\end{align*}

\noi
By equation \eqref{estimate X} and using the definition of $X$,
it follows that
\[\Big|\frac{\dot{\alpha}}{\alpha}-C\Big|\leq c\eps^{1+\frac{2\delta}{3}}.\]

\noi
Thus
\[\frac{\dot{\alpha}}{\alpha}\leq \frac{\eps}{\pi}\int b'(a+\frac{x}{\mu})|\eta(x)|^2\frac{dx}{\mu}
+c\eps^{1+\frac{2\delta}{3}}.\]

\noi
Proceeding as we did for $\mu(t)$ and possibly making the constant
$\tilde{c}_0$ larger, we obtain that
\begin{align*}
|\alpha(t)-\alpha_0|
\leq \tilde{c}_0\dl\eps^{\frac{1}{2}+\delta}\ln(\frac{1}{\eps}),\\
|\alpha(t)-\cj{\alpha}(t)|\leq \tilde{c}_0\dl\eps^{\frac{1}{2}+\delta}
\ln(\frac{1}{\eps}),
\end{align*}

\noi
for $0\leq t\leq \frac{\dl}{6\ln c_0}\cdot\frac{1}{\eps^{\frac{1}{2}-\delta}}
\ln(\frac{1}{\eps})$.

We thus proved that for the above range of time, $\mu(t)$ and $\alpha(t)$
stay close to $\mu_0$ and $\alpha_0$ respectively.
The definition of $X$ \eqref{eqn:X} and the estimate \eqref{estimate X}
then yield that $a,\alpha,\phi,\mu$
satisfy the perturbed effective dynamics \eqref{sys perturbations}
in the statement of Theorem \ref{main theorem}.

By Lemma \ref{lemma:alpha2mu} we have that $\|w\|_{L^2}^2
=\pi\Big(\frac{\alpha_0^2\mu_0}{\alpha^2\mu}-1\Big)$. Then, the
equations satisfied by $\bar{\alpha}$ and $\bar{\mu}$ yield
$\partial_t(\bar{\alpha}^2\bar{\mu})=0$, and thus we obtain that
\begin{align*}
\alpha^2\mu=\alpha_0^2\mu_0+c\eps^{1+\frac{2\delta}{3}},\,\,\,\,\,
\bar{\alpha}^2\bar{\mu}=\alpha_0^2\mu_0.
\end{align*}

\noi
Subtracting the equations satisfied by $\phi$ and $\bar{\phi}$, we then obtain that
\begin{align*}
|\dot{\phi}-\dot{\bar{\phi}}|&=\Big|-\frac{\alpha_0^2\mu_0}{4}(\mu-\bar{\mu})
-\frac{\eps}{\pi}\int\big(b(a+\frac{x}{\mu})-b(\bar{a}+\frac{x}{\bar{\mu}})\Big)|\eta|^2dx\\
&\hphantom{XXXXXXXXXX}-\frac{\eps}{\pi}\int\Big(b'(a+\frac{x}{\mu})-b'(\bar{a}
+\frac{x}{\bar{\mu}})\Big)x|\eta(x)|^2\frac{dx}{\mu}\Big|+c\eps^{1+\frac{2\dl}{3}}\\
&\leq c|\mu-\bar{\mu}|+c\eps
\leq (\tilde{c}_0\delta+c)\eps^{\frac{1}{2}
+\delta}\ln(\frac{1}{\eps}).
\end{align*}

\noi
Integrating, we obtain the desired estimate for $|\phi-\bar{\phi}|$.
Similarly we obtain the estimate for $|a-\bar{a}|$.

Let $0<\rho\ll 1$. Suppose $\eps$ is small enough such that $\eps^{\rho}\ln(\frac{1}{\eps})^2\leq 1$.
Then we have that
\[|\phi-\bar{\phi}|\leq \tilde{c}_0\dl\eps^{2\delta}\ln(\frac{1}{\eps})^2\leq c\eps^{2\delta-\rho}.\]

\noi
If $2\delta-\rho \geq\frac{1}{2}+\frac{\delta}{3}$, which is equivalent
to $\delta\geq \frac{3}{10}+\frac{3}{5}\rho>\frac{3}{10}$,
then one can easily see that
$|\phi-\bar{\phi}|\leq \eps^{\frac{1}{2}+\frac{\delta}{3}}$.
This together with the approximations for $a,\alpha,\mu$
in equations \eqref{eqn a-bar{a}} yields

\begin{align*}
\|\alpha(t)e^{i\phi(t)}\mu(t)\eta(\mu(t)(x-a(t)))
-\bar{\alpha}(t)e^{i\bar{\phi}(t)}\bar{\mu}(t)\eta(\bar{\mu}(t)(x-\bar{a}(t)))\|_{H^{\frac{1}{2}}_+}
\leq c\eps^{\frac{1}{2}+\frac{\delta}{3}}.
\end{align*}

\noi
Thus, if $\delta\geq \frac{3}{10}+\frac{3}{5}\rho>\frac{3}{10}$,
we have that
\begin{align*}
\|u(t)-\bar{\alpha}(t)e^{i\bar{\phi}(t)}\bar{\mu}(t)\eta(\bar{\mu}(t)(x-\bar{a}(t)))\|_{H^{\frac{1}{2}}_+}
\leq c\eps^{\frac{1}{2}+\frac{\delta}{3}},
\end{align*}

\noi
for times $0\leq t\leq \frac{\dl}{6\ln c_0}\cdot\frac{1}{\eps^{\frac{1}{2}-\delta}}\ln(\frac{1}{\eps})$.
\end{proof}

{\bf Acknowledgments:}
The author is grateful to her Ph.D. advisor Prof. Patrick G\'erard
for suggesting this problem to her and for helpful comments on the paper.


\begin{thebibliography}{99}

\bibitem{Bronki Jerrard}J. C. Bronski, R. L. Jerrard,
{\it Soliton dynamics in a potential,} Mathematical Research Letters 7 (2000), 329--42.

\bibitem{F2} J. Fr\"ohlich, S. Gustafson, B. L. G. Jonsson, I. M. Sigal,
{\it Solitary wave dynamics in an external potential,} Communications in Mathematical Physics 250 (2004), 613--42.

\bibitem{F3} J. Fr\"ohlich, S. Gustafson, B. L. G. Jonsson, I. M. Sigal,
{\it Long time motion of NLS solitary waves in a confining potential,} Annales Henri Poincaré 7 (2006), 621--660.

\bibitem{F1} J. Fr\"ohlich, T.-P. Tsai, H.-T. Yau,
{\it On the point-particle (Newtonian) limit of the nonlinear Hartree equation,}  Communications in Mathematical Physics 225 (2002), 223--74.

\bibitem{PGSG} P. G\'erard, S. Grellier,
{\it The cubic Szeg\"{o} equation,} Annales Scientifiques de l'Ecole Normale Sup\'erieure, Paris, $4^e$ s\'erie, t. 43, (2010), 761--810.

\bibitem{PGSGX} P. G\'erard, S. Grellier, {\it L'\'equation de Szeg\"{o} cubique,} S\'eminaire X EDP, 20 octobre 2008, \'Ecole Polytechnique, Palaiseau,
\url{http://sedp.cedram.org/cedram-bin/article/SEDP_2008-2009____A2_0.pdf}


\bibitem{Zworski delta} J. Holmer, M. Zworski,
{\it Slow soliton interaction with delta impurities,} J. Mod. Dyn. 1 (2007), no. 4, 689--718.

\bibitem{Zworski}J. Holmer, M. Zworski,
{\it Soliton interaction with slowly varying potentials,} Int. Math. Res. Not., (2008),  no. 10, Art. ID rnn026, 36 pp.

\bibitem{Zworski-Perelman}J. Holmer, G. Perelman, M. Zworski,
{\it Effective dynamics of double solitons for perturbed mKdV,}
preprint arXiv:0912.5122v2.


\bibitem{Keraani} S. Keraani,
{\it Semiclassical limit of a class of Schrödinger equations with potential,} Comm. Partial Differential Equations 27 (2002), no. 3-4, 693–704.

\bibitem{Keraani II} S. Keraani,
{\it Semiclassical limit for nonlinear Schrödinger equation with potential. II,} Asymptot. Anal. 47 (2006), no. 3--4, 171--186.

\bibitem{Nikolskii} N.K. Nikolskii,
{\it Operators, Functions and Systems: An Easy Reading, Vol.I: Hardy, Hankel, and Toeplitz,} Mathematical Surveys and Monographs, vol.92, AMS, (2002).

\bibitem{pocov1} O. Pocovnicu,
{\it Traveling waves for the cubic Szeg\"{o} equation on the real line,} to appear in Analysis and PDE.
\end{thebibliography}
\end{document}